\documentclass[12pt]{amsart}
\usepackage{pb-diagram}

\usepackage{amscd}
\usepackage{amsfonts}
\usepackage{amsmath}
\usepackage{amssymb}
\usepackage{color}
\usepackage{amsthm}
\usepackage[colorlinks,linkcolor=blue,citecolor=blue]{hyperref} 
\everymath{\displaystyle}

\newcommand{\tr}{\textnormal{tr}}
\newcommand{\Ric}{\textnormal{Ric}}
\newcommand{\RR}{\textnormal{R}}

\newcommand{\kod}{\textnormal{kod}}
\newcommand{\Ker}{\textnormal{Ker}\,}
\newcommand{\rank}{\textnormal{rank}\,}
\renewcommand{\Im}{\textnormal{Im}\,}
\newcommand{\Hess}{\textnormal{Hess\,}}
\newcommand{\Rm}{\textnormal{Rm}}
\newcommand{\II}{\textnormal{II}}
\newcommand{\PT}{\textnormal{PT}}

\newcommand{\cH}{\mathcal{H}}

\newcommand{\dif}{\textnormal{d}}
\newcommand{\dbar}{\overline{\partial}}
\newcommand{\dd}{\partial}

\newcommand{\ddt}[1]{\frac{\partial #1}{\partial t}}
\newcommand{\dds}[1]{\frac{\partial #1}{\partial s}}

\newcommand{\cO}{\mathcal{O}}

\newcommand{\cV}{\mathcal{V}}

\newcommand{\ddbar}{\sqrt{-1}\partial\dbar}

\newcommand{\diam}{\textnormal{Diam}}

\newcommand{\BC}{\mathbb{C}}
\newcommand{\BR}{\mathbb{R}}
\newcommand{\BP}{\mathbb{P}}
\newcommand{\BQ}{\mathbb{Q}}
\newcommand{\BN}{\mathbb{N}}

\newtheorem{thm}{Theorem}[section]
\newtheorem{lemma}[thm]{Lemma}
\newtheorem{coro}[thm]{Corollary}
\newtheorem{prop}[thm]{Proposition}

\numberwithin{equation}{section}

\theoremstyle{definition}
\newtheorem{rmk}[thm]{Remark}

\theoremstyle{definition}
\newtheorem{defi}[thm]{Definition}

\newenvironment{sketch}{%
  \begin{proof}[Sketch of Proof]
}{%
  \end{proof}%
}

\title{Finite-Time Singularities of the K\"ahler--Ricci Flow on a $\BC P^m$-Bundle over a Product of K\"ahler--Einstein Manifolds}
\author{Yifan Xiao}
\thanks{The author is supported in part by Fundamental and Interdisciplinary Disciplines Breakthrough Plan of the Ministry of Education of China (JYB2025XDXM112).}
\address{School of Mathematical Sciences, East China Normal University, Shanghai, 200241, China}
\email{51275500015@stu.ecnu.edu.cn}

\usepackage{fancyhdr}
\begin{document}

\begin{abstract}

In this paper, we study the K\"ahler--Ricci flow on $\mathbb{CP}^m$-bundles over a product of K\"ahler--Einstein manifolds, starting from an initial metric with Calabi symmetry. We prove that every finite-time singularity arising along the flow must be of Type~I.

\end{abstract}

\maketitle
\tableofcontents

\section{Introduction}

The Ricci flow, introduced by Hamilton \cite{Ham82}, is a central tool in geometric analysis, most notably in Perelman's work on the Poincar\'e conjecture \cite{Per02,Per03a,Per03b}. Cao introduced the K\"ahler--Ricci flow in \cite{Cao85}. When the canonical bundle is not nef, this flow becomes singular in finite time, and understanding the rate and geometry of the degeneration is a basic problem in the analytic minimal model program of Song--Tian \cite{ST17}.

Let $T$ be the singular time. Following Hamilton \cite{Ham95b}, the singularity is of Type~I if $|\Rm|\leq C(T-t)^{-1}$ for some constant $C$, and of Type~II otherwise. Type~I blow-ups are closely related to shrinking gradient Ricci solitons \cite{EMT11}. Although Perelman's estimates control scalar curvature and diameter along the Fano K\"ahler--Ricci flow \cite{SeT08}, they do not in general control the full Riemann curvature tensor. Type~II examples with substantial symmetry are known \cite{GZ08,LTZ24}, so symmetry alone does not determine the singularity type.

For symmetric K\"ahler--Ricci flows, the metric can often be described by a radial function satisfying a scalar parabolic equation. Feldman--Ilmanen--Knopf \cite{FIK03} constructed rotationally symmetric shrinking and expanding gradient K\"ahler--Ricci solitons on complex line bundles over projective space, providing noncompact shrinking models for later singularity analysis. Song--Weinkove \cite{SW11} studied invariant flows on Hirzebruch surfaces and established a trichotomy: the manifold shrinks to a point, the projective-line fibers collapse and the flow converges to the base, or the exceptional divisor is contracted; they also obtained higher-dimensional analogues. Fong \cite{Fon14} extended the fiber-collapse part of this analysis to projective-line bundles over a compact K\"ahler--Einstein manifold. Under a cohomological criterion on the bundle and the initial K\"ahler class, he proved Gromov--Hausdorff convergence to the base, a Type~I curvature bound, and a blow-up model whose universal cover is the product of a flat factor and a shrinking projective line. Song \cite{Son15} treated the contraction of the exceptional divisor on the blow-up of projective space at one point and proved that the singularity is of Type~I. Guo--Song \cite{GS16} then identified the corresponding parabolic blow-up limit with the unique rotationally symmetric complete shrinking soliton constructed in \cite{FIK03}, thereby confirming the Feldman--Ilmanen--Knopf conjecture.

Jian--Song--Tian \cite{JST23} obtained estimates for general finite-time K\"ahler--Ricci flows without any symmetry assumption. They introduced the weighted Ricci potential and its minimum point, called a Ricci vertex. Their Li--Yau type gradient and Laplacian estimates give a Type~I scalar curvature bound at every point whose distance from a Ricci vertex remains bounded in the rescaled metric $(T-t)^{-1}\omega(t)$. Combined with a Type~I volume bound on a suitable tubular neighborhood, their Harnack estimates also give diameter and scalar curvature bounds for the fiber containing a Ricci vertex; they verify this volume bound for the Fano fiber bundles considered in their applications. Applied to higher-rank projective bundles over a single K\"ahler--Einstein manifold, these estimates prove that every flow with Calabi symmetry develops a Type~I singularity, both when a submanifold of higher codimension is contracted and when the fibers collapse, and they determine the corresponding smooth Type~I blow-up limits. The estimate near a Ricci vertex is the principal local analytic input in the present paper. Most recently, Fong--Tran \cite{FT26a} kept the projective-line fiber but allowed the base to be a product of K\"ahler--Einstein manifolds, using the Wang--Wang--Dancer ansatz \cite{WW98,DW11}. They proved that this form of the metric is preserved by the Ricci flow and that every finite-time singularity in the K\"ahler case is of Type~I. Their proof obtains a lower bound, in terms of the remaining time, for each coefficient multiplying a metric on a base factor. Under a hypothetical Type~II rescaling, this bound makes the O'Neill tensors \cite{O'N66} and the curvature entirely tangent to the base tend to zero, so the universal cover of the limit splits into a flat factor and a two-dimensional nonflat factor. The latter must be the cigar soliton, contradicting Perelman's noncollapsing theorem. This last two-dimensional contradiction already occurs in \cite{Fon14}; \cite{FT26a} supplies the estimates needed to carry it out when the base has several factors.

We study the singularity type of the K\"ahler--Ricci flow on the $\BC P^m$-bundle $\pi:X\to N$ over $N=N_1\times\cdots\times N_r$, starting from a Calabi-symmetric metric, where the factors $N_k$ are compact K\"ahler--Einstein manifolds. Our main result is the following.

\begin{thm}\label{thm}
Suppose that $(N_k^{n_k}, \omega_k)$, $k=1,\ldots,r$, are compact K\"ahler--Einstein manifolds satisfying $\Ric(\omega_k)=\lambda_k\omega_k$. Let $N=\prod_{k=1}^r N_k$, and let $(X,\omega_0)$, constructed in Section~\ref{ansatz}, be the $\BC P^m$-bundle over $N$ with a Calabi-symmetric K\"ahler metric. Then the K\"ahler--Ricci flow starting from $\omega_0$ develops a finite-time Type~I singularity.

\end{thm}

The case $m=1$ was proved in \cite{FT26a}. When $m\geq2$, however, the fiber contains directions tangent to $\BC P^{m-1}$ in addition to the radial complex direction. This prevents a direct use of either \cite{JST23} or \cite{FT26a}. In the bundle setting considered in \cite{JST23}, the base and fiber coefficients are linked by one radial function, whereas the coefficients of the factors $N_k$ evolve independently here. Moreover, the Ricci-vertex estimate is local in the rescaled metric and controls only scalar curvature; it neither reaches points far from the zero section in the contraction case nor directly bounds the separate components of the Riemann curvature tensor. The argument of \cite{FT26a} relies on a lower bound for each coefficient multiplying a metric on a base factor. In a higher-dimensional fiber, the additional coefficient in the $\BC P^{m-1}$ directions vanishes at the zero section, so a blow-up based near that section need not reduce immediately to two nonflat dimensions.

Our proof obtains the missing componentwise curvature bounds and controls the possible extra nonflat directions by using the radial coordinate $x$ and the rescaled quantity $y=x/(T-t)$. When $y$ remains bounded, we choose the closed form in the weighted Ricci potential so that a Ricci vertex stays within bounded rescaled distance of the zero section. The estimate of \cite{JST23} then controls scalar curvature, and the explicit curvature identities and the smoothness conditions at the zero section give bounds for every component of the Riemann curvature tensor, including those involving different factors $N_k$. In the fiber-collapse case, the bounded rescaled diameter of each fiber extends this control to all points. In the contraction case, points with $y\to\infty$ require a different argument: estimates for the projection onto the base directions and for curvature entirely tangent to the base show that these directions become parallel and flat under a Type~II rescaling. The universal cover of a smooth blow-up limit then splits into a flat factor and a two-dimensional nonflat factor, so the contradiction between the cigar soliton and noncollapsing used in \cite{Fon14,FT26a} applies.

Recent work of Fong and Tran \cite{FT26b} establishes Type I behavior and identifies singularity models in a broad circle-bundle ansatz setting; their Type I theorem includes the bundles considered here. Their Type I proof derives a lower bound for fibrewise holomorphic bisectional curvature and combines splitting of hypothetical Type II limits with rigidity of noncollapsed steady Kähler–Ricci solitons with nonnegative holomorphic bisectional curvature. Meanwhile, Jian and Song \cite{JS26} recently prove global Type I curvature bounds and cylindrical tangent-flow limits for projective-line bundles whose limiting class is the pullback of a Kähler class on the base, without any symmetry assumption. The present work was developed independently of these two recent preprints. Our proof follows a different strategy: using the Ricci-vertex estimates of \cite{JST23}, we recover full curvature bounds on $y$-bounded regions and then use the argument similar as \cite{FT26a} and show that any hypothetical Type II blow-up in the complementary region splits, on its universal cover, into a flat factor and a two-dimensional factor. The latter is ruled out by the incompatibility of the cigar soliton with noncollapsing.

This paper is organized as follows. Section~\ref{pre} recalls the weighted Ricci potential, the Ricci vertex, and the estimates from \cite{JST23}. Section~\ref{ansatz} constructs the bundles and metrics, computes their connection and curvature, reduces the K\"ahler--Ricci flow to equations for the radial function and the base coefficients, and establishes the basic estimates. Section~\ref{proof} treats the regions where $y$ is bounded and unbounded, and proves Theorem~\ref{thm} in the fiber-collapse and contraction cases.

\paragraph*{\bfseries{Acknowledgement}}

The author thanks Professor Linfeng Zhou for his continued encouragement and support. He also thanks the geometry group at East China Normal University for their support and the seminar participants for helpful discussions. The author would also like to acknowledge the assistance of AI with spelling and grammar checks and with the pre-submission review. The main ideas, mathematical statements, hypotheses, references, proofs, and conclusions developed in this paper are the author's own. The author takes full responsibility for the contents of this paper.

\section{Preliminaries}\label{pre}

Jian, Song, and Tian \cite{JST23} introduced the Ricci vertex and established the estimates near it that will be used below. We recall the main definition and lemmas.

\subsection{Ricci Vertex}

Suppose that $X^n$ is a compact K\"ahler manifold with a K\"ahler metric $\omega_0$ whose class is rational, $[\omega_0]\in H^{1,1}(X,\BR)\cap H^2(X,\BQ)$. The K\"ahler--Ricci flow is given by\[\begin{cases}
\frac{\dd \omega}{\dd t}=-\Ric(\omega);\\
\omega(0)=\omega_0.
\end{cases}\]
Assume that the maximal existence time of the solution is
\[
T=\sup\{t>0: [\omega_0]-tc_1(X)>0\}<+\infty
\].

Since $[\omega_0]\in H^2(X,\BQ)$, the limiting cohomology class $\zeta=[\omega_0]-Tc_1(X)$ is semi-ample by Kawamata's base point free theorem \cite{Kaw85}. It induces a surjective holomorphic map\[
\Phi: X\to \widetilde{X}\subset \BC P^M,
\]
where $\widetilde{X}$ is a normal projective variety and $\dim \widetilde{X}=\kod\, \zeta \leq n$. Then there exists a closed $(1,1)$-form $\theta_{\widetilde{X}}$ on $\widetilde{X}$ satisfying $\zeta=[\Phi^*\theta_{\widetilde{X}}]\in H^{1,1}(X,\BR)$. For brevity, we shall abuse notation and write $\theta\in (\Phi^*)^{-1}\zeta$ to indicate $\zeta=[\Phi^*\theta]$, noting that $\Phi^*$ is generally not invertible on cohomology.

Notice that $\theta_{\widetilde{X}}$ is not unique. Smooth forms on the normal variety $\widetilde{X}$ are understood locally as restrictions of ambient smooth forms under local embeddings into some $\BC^{M'}$; thus $\theta_{\widetilde{X}}$ can be chosen in this sense.

By the $\dd\dbar$-lemma, there exists $u\in C^\infty(X\times[0,T))$ such that\[
\Ric(\omega)-\frac{1}{T-t}\omega=-\frac{1}{T-t}\Phi^*\theta_{\widetilde{X}}-\ddbar u.
\]

For normalization, we require that $\min_{X}u(\cdot,t)=1$. With this normalization, $u$ is unique and depends only on the choice of $\theta_{\widetilde{X}}$.

\begin{defi}
The function $u$ defined above is called the \textbf{weighted Ricci potential} associated with $\theta_{\widetilde{X}}$, and a minimum point (not necessarily unique) of $u(\cdot,t)$ is called the \textbf{Ricci vertex} associated with $\theta_{\widetilde{X}}$ at $t\in [0,T)$.
\end{defi}

Set $s=\log\frac{T}{T-t}$ and $\widetilde{\omega}(s)=\frac{1}{T-t}\omega(T-Te^{-s})$. Then $\widetilde{\omega}$ satisfies the normalized K\"ahler--Ricci flow\[\begin{cases}
\dds{\widetilde\omega}=-\Ric(\widetilde\omega)+\widetilde\omega;\\
\widetilde\omega(0)=\widetilde\omega_0=\frac{1}{T}\omega_0,
\end{cases}\]
and the solution $\widetilde{\omega}(s)$ exists for all $s\in [0,+\infty)$. With this change of variables, \[
[\widetilde{\omega}(s)]=c_1(X)+\frac{e^s}{T}\zeta,
\]
and the weighted Ricci potential satisfies
\[
\Ric(\widetilde{\omega})-\widetilde{\omega}=-\frac{e^s}{T}\Phi^*\theta_{\widetilde{X}}-\ddbar u.
\]

\subsection{Some Useful Estimates}

We state some useful estimates in this section. We include only sketches here; the complete proofs can be found in \cite{JST23}.

The following lemma establishes a Li--Yau type estimate for the weighted Ricci potential $u$.

\begin{lemma}
There exists $C=C(n,\omega_0,\theta_{\tilde{X}})>0$ such that\[
\frac{|\Delta u|+|\nabla u|^2}{u}\leq \frac{C}{T-t}
\]
on $X\times [0,T)$.
\end{lemma}

\begin{sketch}
It suffices to work with the normalized flow. Set $s=\log \frac{T}{T-t}$ and $\widetilde{\omega}(s)=(T-t)^{-1}\omega(t)$. Before imposing the minimum normalization, denote the corresponding weighted Ricci potential by $\widehat{u}$ and let $a(s)=\inf_X\widehat{u}(\cdot,s)$, so that the potential used in the statement is $u=\widehat{u}-a+1$. Fix $s_0>0$. The elementary monotonicity formulas for $e^{-s}(a(s)-B_0)$ and $e^{-s}(a(s)+B_0)$ imply that, for a uniform constant $B$, the function
$$
b(s)=e^{s-s_0}a(s_0)-B,\qquad 0\leq s\leq s_0,
$$
satisfies $b\leq a$, $b'=b+B$, and $|b(s)|\leq Ce^s$. Thus
$$
w=\widehat{u}-b+1
$$
is positive and bounded below by one.

Write $\mathcal{E}_s$ for the projective background term contributed by $\Phi^*\theta_{\tilde{X}}$ in the normalized Ricci-potential equation. The evolution equations for $\widehat{u}$ and
$$
K_0=-\Delta_{\widetilde{\omega}}\widehat{u}
+\operatorname{tr}_{\widetilde{\omega}}\mathcal{E}_s
=R_{\widetilde{\omega}}-n
$$
are combined with the parabolic Schwarz lemma. More precisely, after adding to $|\nabla\widehat{u}|^2/w$ suitable trace terms associated with positive projective forms dominating $\mathcal{E}_s$, one obtains a quantity $H$ satisfying
$$
\left(\partial_s-\Delta_{\widetilde{\omega}}\right)H
\leq
-c\frac{|\nabla\widehat{u}|^4}{w^3}
+C\frac{|\nabla\widehat{u}|^2}{w^2}
+\frac{2(1-\varepsilon)}{w}
\operatorname{Re}
\left\langle\nabla H,\nabla\widehat{u}\right\rangle
+Ce^{-s}.
$$
At a space-time maximum of $H$, the drift term vanishes. Since $w\leq Ce^s$, the maximum principle and the Schwarz estimate give
$$
\frac{|\nabla\widehat{u}|^2}{w}\leq C.
$$
For the Laplacian estimate, one applies the same argument to
$$
\mathcal{K}
=
\frac{K_0+C_0}{w}+A H,
$$
where $C_0$ and $A$ are chosen sufficiently large. The evolution equation for $K_0$, the preceding gradient bound, and the Schwarz lemma yield
$$
\left(\partial_s-\Delta_{\widetilde{\omega}}\right)\mathcal{K}
\leq
-c\frac{|\nabla\overline{\nabla}\widehat{u}|^2}{w}
+\frac{2}{w}
\operatorname{Re}
\left\langle\nabla\mathcal{K},\nabla\widehat{u}\right\rangle
+C.
$$
The maximum principle therefore controls $|\Delta_{\widetilde{\omega}}\widehat{u}|/w$. The opposite sign follows from the uniform scalar-curvature lower bound and the parabolic Schwarz estimate, and hence
$$
\frac{|\Delta_{\widetilde{\omega}}\widehat{u}|
      +|\nabla\widehat{u}|^2}{w}
\leq C
$$
on $X\times[0,s_0]$. At time $s_0$ one has $w=u+B\leq (B+1)u$, and $s_0$ is arbitrary. Consequently,
$$
\frac{|\Delta_{\widetilde{\omega}}u|
      +|\nabla u|_{\widetilde{\omega}}^2}{u}
\leq C.
$$
Rescaling back to $\omega(t)$ gives the stronger estimate
$$
\frac{|\Delta_{\omega(t)}u|
      +|\nabla u|_{\omega(t)}^2}{u}
\leq \frac{C}{T-t},
$$
which implies the asserted inequality.

\end{sketch}

The following lemma gives the local Type~I scalar curvature bound near the Ricci vertex.

\begin{lemma}\label{JST1.3}
For any Ricci vertex $p$ at $t\in[0,T)$ associated with $\theta_{\tilde{X}}$, there exists $C=C(n,\omega_0,\theta_{\tilde{X}})>0$ such that
\[
(T-t)|\RR(x,t)|\leq C\left(1+\frac{d^2_{\omega(t)}(x,p)}{T-t}\right)
\]
on $X\times [0,T)$.

In particular, fix $t\in[0,T)$. Then for any $\delta>0$, \[
|\RR(x,t)|\leq \frac{C(1+\delta^2)}{T-t}
\]
on $B_{\omega(t)}(p;\delta\sqrt{T-t})$, where $B_{\omega}(x;r)$ denotes the geodesic ball centered at $x\in M$ of radius $r$ with respect to the metric $\omega$.\label{JST13}
\end{lemma} 

\begin{sketch}

Work again with the normalized metric $\widetilde{\omega}(s)=(T-t)^{-1}\omega(t)$. Let $p$ be a Ricci vertex at time $t$. By the chosen normalization, $u(p,t)=1$. The gradient estimate above gives
$$
|\nabla\sqrt{u}|_{\widetilde{\omega}}
=
\frac{|\nabla u|_{\widetilde{\omega}}}{2\sqrt{u}}
\leq C.
$$
Integrating this inequality along a minimizing $\widetilde{\omega}(s)$-geodesic from $p$ to $x$ yields
$$
\sqrt{u(x,t)}
\leq
1+C\,d_{\widetilde{\omega}(s)}(x,p),
$$
and therefore
$$
u(x,t)
\leq
C\left(1+d_{\widetilde{\omega}(s)}^2(x,p)\right).
$$

Taking the trace of the defining equation for the weighted Ricci potential gives
$$
R_{\widetilde{\omega}}
=
n-\Delta_{\widetilde{\omega}}u
-\operatorname{tr}_{\widetilde{\omega}}
\left((T-t)^{-1}\Phi^*\theta_{\tilde{X}}\right).
$$
The parabolic Schwarz lemma uniformly controls the last term, while the Laplacian estimate gives $|\Delta_{\widetilde{\omega}}u|\leq Cu$. It follows that
$$
|R_{\widetilde{\omega}}(x,s)|
\leq
C\left(1+d_{\widetilde{\omega}(s)}^2(x,p)\right).
$$
Finally, since
$$
R_{\widetilde{\omega}}=(T-t)R_{\omega(t)},
\qquad
d_{\widetilde{\omega}(s)}=(T-t)^{-1/2}d_{\omega(t)},
$$
we obtain
$$
(T-t)|R(x,t)|
\leq
C\left(
1+\frac{d_{\omega(t)}^2(x,p)}{T-t}
\right).
$$
In particular, if $x\in B_{\omega(t)}(p;\delta\sqrt{T-t})$, then
$$
|R(x,t)|
\leq
\frac{C(1+\delta^2)}{T-t}.
$$

\end{sketch}

\section{Ansatz}\label{ansatz}

\subsection{$\BC P^m$-Bundle and Calabi-Symmetric Metrics}

Let $L_k$ be a holomorphic line bundle over $N_k$ satisfying
\[
c_1(L_k)=-q_k[\omega_k], \quad q_k\in \BQ\setminus\{0\}
\]
for each $k$, and let $L=\bigotimes\limits_{k=1}^r \pi_{N_k}^*L_k$, where $\pi_{N_k}:N\to N_k$ is the projection. Choose Hermitian metrics $h_k$ on $L_k$ such that the Chern curvature satisfies
\[
F_{h_k}(L_k)=-\ddbar \log h_k=-q_k\omega_k.
\]
Then
\[
\pi: X=\BP(\cO_{N}\oplus L^{\oplus m})\to N
\]
is a $\BC P^m$-bundle over $N$. In this paper, we assume that $m\geq 2$. Denote $n=\dim N=\sum_{k=1}^r n_k$.

Define the Hermitian metric $h=\bigotimes\limits_{k=1}^r \pi_{N_k}^*h_k$ on $L$ and the radial coordinate function
\[
\rho=\log (h(z)|\xi|^2),
\]
where $z$ denotes local holomorphic coordinates on $N$ and $\xi=(\xi^1, \cdots, \xi^m)$ denotes fiber coordinates on $L^{\oplus m}$.

We write a Calabi $U(m)$-symmetric K\"ahler metric \cite{Cal82} in the form
\begin{equation}
\omega=\sum_{k=1}^r a_k\omega_k+\ddbar \varphi(\rho),\label{calabi}
\end{equation}
where $a_k\in \BR$ and $\varphi\in C^{\infty}(\BR)$ satisfy certain conditions. Notice that the form $\omega$ defined above is always closed.

Let \(P_0=\BP(\cO_N)\) be the zero section and \(D_\infty=\BP(L^{\oplus m})\) the divisor at infinity in $X$. Then there is a natural free $S^1$-action \((z,\xi)\mapsto (z,e^{i\theta}\xi)\) on $X_0=X\setminus(P_0\cup D_\infty)$, and
\[
(X\setminus(P_0\cup D_\infty))/S^1\cong I\times(\BC P^{m-1}\times N),
\]
where $I$ is an open interval. This action induces a principal $S^1$-bundle structure over $I\times(\BC P^{m-1}\times N)$, and there is a connection $1$-form $\eta$ on this principal $S^1$-bundle satisfying
\begin{equation}
\dif \eta=-\omega_{FS}-\sum_{k=1}^r q_k\omega_k=-\ddbar \rho,\label{eta}
\end{equation}
where $\omega_{FS}$ is the Fubini--Study form on $\BC P^{m-1}$. Indeed, $\eta$ is induced by the connection of the Hermitian metric $h$ on $L^{\oplus m}$, which is locally written as
\[
\nabla \xi^\alpha =\dif \xi^\alpha+h^{-1}\dd h\xi^\alpha.
\]

Let $H=\cO_X(1)=\cO_X(D_\infty)$. Then the K\"ahler class is given by
\begin{equation}
[\omega]=Bc_1(H)+\pi^*\sum_{k=1}^r a_k[\omega_k],\label{class}
\end{equation}
If $B,a_k\in \BQ$, then $[\omega]\in H^2(X,\BQ)$.

Set
\[
x=\varphi_\rho, \Theta=x_\rho=\varphi_{\rho\rho}, d_k(x)=a_k+q_kx, 
\]
Then the Riemannian metric associated with $\omega$ is given by
\[
g=\frac{1}{\Theta}\dif x^2+\Theta \eta^2+xg_{FS}+\sum_{k=1}^r d_k(x)g_k,
\]
where $g_k$ and $g_{FS}$ are the Riemannian metrics associated with $\omega_k$ and $\omega_{FS}$, respectively.

The following criterion is due to Calabi \cite{Cal82}.

\begin{prop}
$\omega$ as defined above is a K\"ahler metric if and only if

$(1)$ Positivity condition: $a_k,d_k(x)>0$ for each $k=1,\cdots, r$, and $x=\varphi_\rho>0, \Theta=x_{\rho}=\varphi_{\rho\rho}>0$ for all $\rho\in(-\infty,+\infty)$;

$(2)$ Closing condition at $P_0$: $\psi_0(\sigma)=\varphi(\log \sigma)$ extends smoothly to $\sigma=0$ and $\psi_0'(0)>0$.

$(3)$ Closing condition at $D_\infty$: $B=\lim\limits_{\rho\to \infty} x>0$ is finite, and $\psi_\infty(\sigma)=\varphi(-\log \sigma)+B\log \sigma$ extends smoothly to $\sigma=0$ and $\psi_\infty'(0)>0$.

In particular, if $\omega$ is a K\"ahler metric, then $(1)$ implies that

$(4)$ Endpoint condition: \(b_k=d_k(B)=a_k+q_kB>0\);\\
and $(2)$ and $(3)$ imply that

$(5)$ Boundary conditions: \(\Theta(0)=\Theta(B)=0; \Theta_x(0)=1,  \Theta_x(B)=-1.\)\label{kahler}
\end{prop}

In fact, $x\in [0,B]$ describes the "distance" to the zero section $P_0$. Precisely,
\[
P_0=\{x=0\},
\qquad
D_\infty=\{x=B\},
\quad
X_0=\{0<x<B\}.
\]

The connection form \(\eta\) determines the horizontal splitting and induces the orthogonal decomposition on $TX_0$. Define 
\[
\cV=
\operatorname{span}\{\partial_x,\varsigma\}, \cH=\cV^{\perp_g}.
\]
Here $\varsigma$ is the Reeb vector field of $\eta$ satisfying $\eta(\varsigma)=1$.

Denote by \[
\pi_0: X_0\to \BC P^{m-1},\,\pi_k:X_0\to N_k,\, k=1,\cdots,r\]
the projections. Then $\cH$ decomposes as
\[
\cH=\mathcal H_0\oplus\bigoplus_{k=1}^{r}\mathcal H_k,
\]
where \(\mathcal H_0=(\Ker(\pi_0)_*)^{\perp_g}\) is the horizontal lift of
\(T\mathbb{CP}^{m-1}\), while \(\mathcal H_k=(\Ker(\pi_k)_*)^{\perp_g}\) is the
horizontal lift of \(TN_k\). 
Finally
\[
TX_0=\mathcal V\oplus\bigoplus_{k=0}^{r}\mathcal H_k.
\]
It is easy to check that this orthogonal decomposition is \(J\)-invariant, where $J$ is the complex structure on $X$.

\subsection{Computation of Geometric Quantities}

The first Chern class is given by\begin{equation}
c_1(X)=(m+1)c_1(H)+\pi^*\sum_{k=1}^r(\lambda_k-mq_k)[\omega_k].\label{chern}
\end{equation}

Denote $p(x)=x^{m-1}\left(\prod_{k=1}^r d_k^{n_k}\right)$. Straightforward computation shows that the volume form induced by $\omega$ is given by\begin{align*}
\frac{\omega^{m+n}}{(m+n)!}&=p\Theta(\sqrt{-1}\dd\rho\wedge\dd\bar{\rho})\wedge \frac{\omega_{FS}^{m-1}}{(m-1)!}\wedge\left(\bigwedge_{k=1}^r \frac{\omega_k^{n_k}}{n_k!} \right)\\
&=h^m e^{-m\rho}p\Theta\left(\bigwedge_{\alpha=1}^{m}\sqrt{-1}\,\dif\xi^\alpha\wedge \dif\bar\xi^\alpha\right)\wedge\prod_{k=1}^{r}\frac{\omega_k^{n_k}}{n_k!},
\end{align*}

Then the Ricci form is given by

\[
\begin{aligned}
\operatorname{Ric}(\omega)
&=-\sqrt{-1}\partial\bar\partial\log\left(h^m e^{-m\rho}p\Theta\prod_{k=1}^{r}\det g_k\right)\\
&=-\sqrt{-1}\partial\bar\partial\log(p\Theta)+m\sqrt{-1}\,\partial\bar\partial\rho-m\sqrt{-1}\,\partial\bar\partial\log h+\sum_{k=1}^{r}\operatorname{Ric}(\omega_k)\\
&=-\sqrt{-1}\partial\bar\partial\log(p\Theta)+m\omega_{\mathrm{FS}}+\sum_{k=1}^{r}\lambda_k\omega_k\\
&=-\ddbar(\log{p\Theta}-m\rho)+\sum_{k=1}^{r} (\lambda_k-mq_k)\omega_k,
\end{aligned}
\]
where the last equality follows from \eqref{eta}.

By taking the trace of the Ricci form, we obtain the scalar curvature\[\begin{aligned}
\RR&=-\Delta_{\omega} \log(p\Theta)+m\tr_\omega(\omega_{FS})+\sum_{k=1}^r\lambda_k\tr_\omega(\omega_k)\\
&=-\frac{(p\Theta)_{xx}}{p}+\sum_{k=1}^r\frac{n_k\lambda_k}{d_k}+\frac{m(m-1)}{x}.\\
\end{aligned}\]

We next compute the Riemann curvature tensor and the O'Neill tensors associated with the submersion $\widetilde\pi:X_0\to \BC P^{m-1}\times N$. We first introduce an orthonormal coframe to simplify the curvature formulas.

In the remainder of this subsection, we adopt the conventions
\[
(N_0,\omega_0)=(\BC P^{m-1},\omega_{FS}),\quad q_0=1,\quad n_0=m-1,\quad d_0=x.
\]
Here \(\omega_0\) is static and should not be confused with the initial metric of the K\"ahler--Ricci flow.

For $k=0,\cdots,r$, let $J^k$ be the complex structure of $N_k$ and let $\{\bar{\theta}^{k,\alpha}\}_{\alpha=1}^{2n_k}$ be a local orthonormal coframe such that $\omega_k=\frac{1}{2}J^k_{\alpha\beta}\bar{\theta}^{k,\alpha}\wedge\bar{\theta}^{k,\beta}$. Define the $s$-coordinate by $\dif s=\frac{1}{\sqrt{\Theta}}\dif x$. Then the coframe $\{\theta^s, \theta^\eta, \theta^{k,\alpha}: k=0,\cdots,r; \alpha=1,\cdots,2n_k\}$ defined by\[
\theta^s=\dif s,\, \theta^\eta=\sqrt{\Theta}\eta,\, \theta^{k,\alpha}=\sqrt{d_k}\,\bar{\theta}^{k,\alpha}
\] is a local orthonormal coframe for $g$. Denote by $\{e_s,e_\eta,e_{k,\alpha}: k=0,\cdots,r; \alpha=1,\cdots,2n_k\}$ the local frame dual to this coframe.

We denote by $\varpi=\{\varpi_{AB}\}$ and $\Omega=\{\Omega_{AB}\}$ the connection $1$-form and curvature $2$-form, respectively. They are defined by Cartan's structure equations\[
\dif \theta^A=-\varpi_{AB}\wedge \theta^{B},\quad\Omega_{AB}=\dif\varpi_{AB}+\varpi_{AC}\wedge\varpi_{CB}=\frac{1}{2}\RR_{ABCD}\theta^{C}\wedge\theta^{D},
\]
where the indices $A,B,C,D$ take values in $\{s,\eta\}\cup\{(k,\alpha):k=0,\cdots,r;\alpha=1,\cdots,2n_k\}$.

\begin{prop}\label{connection}
For all $k,j=0,\cdots,r$, with the Greek indices ranging over the dimensions corresponding to their factor indices,
\[
\varpi_{\eta s}=\frac{\Theta_x}{2\sqrt{\Theta}}\theta^{\eta},\quad \varpi_{(k,\alpha) s}=\frac{q_k\sqrt{\Theta}}{2d_k}\theta^{k,\alpha},\quad \varpi_{\eta (k,\alpha)}=-\frac{q_k\sqrt{\Theta}}{2d_k}J^k_{\alpha\beta}\theta^{k,\beta},
\]\[
\varpi_{(k,\alpha)(k,\beta)}=\varpi^{k}_{\alpha\beta}+\frac{q_k\sqrt{\Theta}}{2d_k}J^k_{\alpha\beta}\theta^{\eta},
\]
where $\varpi^k$ is the connection $1$-form on $(N_k,\omega_k)$. All other cases are either zero or determined by skew-symmetry.
\end{prop}

\begin{proof}
Define $H=\sqrt{\Theta}$ and $f_k=\sqrt{d_k}$. Then

$(1)$ $\dif\theta^{s}=\dif^2 s=0$;

$(2)$ \(\dif\theta^\eta=\dif(H\eta)=H'\dif s\wedge \eta+H\dif \eta=\frac{\Theta_x}{2\sqrt{\Theta}}\theta^s\wedge\theta^\eta-\sum_{k=0}^r \frac{q_k\sqrt{\Theta}}{2d_k}J^k_{\alpha\beta}\theta^{k\alpha}\wedge\theta^{k\beta};\)

$(3)$ \(\dif \theta^{k,\alpha}=\dif (f_k\bar\theta^{k,\alpha})=f'_k\dif s\wedge \bar\theta^{k,\alpha}+f_k\dif\bar\theta^{k,\alpha}=\frac{q_k\sqrt{\Theta}}{2{d_k}}\theta^s\wedge\theta^{k,\alpha}-\varpi^k_{\alpha\beta}\wedge\theta^{k,\beta}.$

A direct computation using Cartan's structure equations gives the conclusion.
\end{proof}

A straightforward computation using Cartan's structure equations gives the curvature $2$-forms and hence the Riemann curvature tensor.

\begin{coro}\label{Rm}
For all $k,j=0,\cdots,r$, with the Greek indices ranging over the dimensions corresponding to their factor indices,
\begin{align*}
\Omega_{\eta s}&=\frac{1}{2}\Theta_{xx}\theta^s\wedge\theta^\eta+\sum_{k=0}^{r}P_kJ^k_{\alpha\beta}\theta^{k,\alpha}\wedge\theta^{k,\beta},\\
\Omega_{(k,\alpha) s}&=P_k(-\theta^s\wedge\theta^{k,\alpha}+J^k_{\alpha\beta}\theta^\eta\wedge\theta^{k,\beta}),\\
\Omega_{\eta(k,\alpha)}&=P_k(\theta^\eta\wedge\theta^{k,\alpha}+J^k_{\alpha\beta}\theta^s\wedge\theta^{k,\beta}),\\
\Omega_{(k,\alpha)(k,\beta)}&=\frac{1}{2d_k}\RR^k_{\alpha\beta\gamma\delta}\theta^{k,\gamma}\wedge\theta^{k,\delta}-2P_kJ^k_{\alpha\beta}\theta^s\wedge\theta^\eta\\
&-Q_{kk}(\theta^{k,\alpha}\wedge\theta^{k,\beta}+J^k_{\alpha\gamma}J^k_{\beta\delta}\theta^{k,\gamma}\wedge\theta^{k,\delta})-\sum_{j=0}^r Q_{kj} J^k_{\alpha\beta}J^j_{\gamma\delta}\theta^{j,\gamma}\wedge\theta^{j,\delta}\\
\Omega_{(k,\alpha)(j,\beta)}&=-Q_{kj}(\theta^{k,\alpha}\wedge\theta^{j,\beta}+J^k_{\alpha\gamma}J^j_{\beta\delta}\theta^{k,\gamma}\wedge\theta^{j,\delta}) \, (\text{Here } k\neq j),\\
\end{align*}
where $P_k=\frac{q_k^2\Theta}{4d_k^2}-\frac{q_k\Theta_x}{4d_k}$, $Q_{kj}=\frac{q_kq_j\Theta}{4d_kd_j}$, and $\RR^k$ is the Riemann curvature tensor of $(N_k,g_k)$. All other cases are zero.
\end{coro}

\begin{coro}For all $k,j=0,\cdots,r$, with the Greek indices ranging over the dimensions corresponding to their factor indices,
\begin{align*}
\RR_{\eta s s\eta}&=\frac12\Theta_{xx},\\
\RR_{(k,\alpha)s\,s(k,\beta)}&=-P_k\delta_{\alpha\beta},\quad\RR_{(k,\alpha)s\,\eta(k,\beta)}=P_kJ^k_{\alpha\beta},\quad \RR_{\eta(k,\alpha)\,\eta(k,\beta)}=P_k\delta_{\alpha\beta},\\
\RR_{(k,\alpha)(k,\beta)(k,\gamma)(k,\delta)}&=\frac{1}{d_k}\,\RR^k_{\alpha\beta\gamma\delta}-Q_{kk}\Bigl(\delta_{\alpha\gamma}\delta_{\beta\delta}-\delta_{\alpha\delta}\delta_{\beta\gamma}+J^k_{\alpha\gamma}J^k_{\beta\delta}-J^k_{\alpha\delta}J^k_{\beta\gamma}+2J^k_{\alpha\beta}J^k_{\gamma\delta}\Bigr),\\
\RR_{(k,\alpha)(j,\beta)(k,\gamma)(j,\delta)}&=-Q_{kj}\Bigl(\delta_{\alpha\gamma}\delta_{\beta\delta}+J^k_{\alpha\gamma}J^j_{\beta\delta}\Bigr), k\ne j.
\end{align*}

All other components are either zero or determined by the Bianchi symmetries \[
\RR_{ABCD}=-\RR_{ABDC}=-\RR_{BACD}=\RR_{CDAB},
\]
\end{coro}

Define the O'Neill tensors by\[\begin{aligned}
A_X Y&= P_{\cV}\nabla_{P_{\cH} X}P_{\cH} Y+P_{\cH}\nabla_{P_{\cH} X} P_{\cV} Y,\\
T_X Y&= P_{\cH}\nabla_{ P_{\cV} X}P_{\cV} Y+P_\cV\nabla_{ P_{\cV} X}P_{\cH} Y,
\end{aligned}\]
where \(P_{\cH}:TX_0\to \cH\) and \(P_{\cV}:TX_0\to \cV\) are the orthogonal projections. For horizontal vectors $X,Y$ and a unit vertical vector $V$, define
\[
\II_{V}(X,Y)=\langle\nabla_X V,Y\rangle V
\].
\begin{coro}\label{oneill}

For $k,j=0,\dots,r$, with $\alpha=1,\dots,2n_k$ and $\beta=1,\dots,2n_j$,

$(1)$ $T\equiv 0$, i.e. the fibers of $\widetilde\pi:X_0\to \BC P^{m-1}\times N$ are all totally geodesic. Furthermore, $[e_{s},e_{\eta}]=-\dfrac{\Theta_{x}}{2\sqrt{\Theta}}\,e_{\eta}\in\mathcal{V}$.

$(2)$ $A_{e_{k,\alpha}}e_{k,\beta}=\mathrm{II}_{e_{\eta}}(e_{k,\alpha},e_{k,\beta})=\frac{q_{k}\sqrt{\Theta}}{2d_{k}}\,J^{k}_{\alpha\beta}\,e_{\eta},$ and $A_{e_{k,\alpha}}e_{j,\beta}=\mathrm{II}_{e_{\eta}}(e_{k,\alpha},e_{j,\beta})=0$ whenever $k\neq j$.

$(3)$ $\mathrm{II}_{e_{s}}(e_{k,\alpha},e_{k,\beta})=\frac{q_{k}\sqrt{\Theta}}{2d_{k}}\,\delta_{\alpha\beta}\,e_{s}$ and $\mathrm{II}_{e_{s}}(e_{k,\alpha},e_{j,\beta})=0$ whenever $k\neq j$.

$(4)$ $|A|^{2}=|\mathrm{II}_{e_s}|^{2}=\sum_{k=0}^{r}\frac{n_{k}q_{k}^{2}\Theta}{2d_{k}^{2}}.$

\end{coro}

\begin{proof}
Recall the following formula for reading off Lie brackets from Cartan's structure equations:
for any indices $A,B,C\in\{s,\eta\}\cup\{(k,\alpha):k=0,\cdots,r;\alpha=1,\cdots,2n_k\}$,
\begin{equation}\label{eq:LB}
\theta^{C}([e_{A},e_{B}])=\varpi_{CB}(e_{A})-\varpi_{CA}(e_{B}).
\end{equation}

$(1)$ Computation of $T$. Since
\[
\varpi_{(k,\alpha)\eta}=\frac{q_{k}\sqrt{\Theta}}{2d_{k}}J^{k}_{\alpha\beta}\,\theta^{k,\beta},
\qquad
\varpi_{(k,\alpha)s}=\frac{q_{k}\sqrt{\Theta}}{2d_{k}}\,\theta^{k,\alpha}.
\]
and $\theta^{k,\gamma}(e_{s})=\theta^{k,\gamma}(e_{\eta})=0$, formula \eqref{eq:LB} gives
$\theta^{k,\alpha}([e_{s},e_{\eta}])=0$ for every $k,\alpha$. Hence $P_{\cH}[e_{s},e_{\eta}]=0$ and $T\equiv0$.

Furthermore, 
\[
\theta^{s}([e_{s},e_{\eta}])
=\varpi_{s\eta}(e_{s})-\varpi_{ss}(e_{\eta})=0,
\]
\[
\theta^{\eta}([e_{s},e_{\eta}])
=\varpi_{\eta\eta}(e_{s})-\varpi_{\eta s}(e_{\eta})
=-\frac{\Theta_{x}}{2\sqrt{\Theta}},
\]
where we have used $\varpi_{\eta s}=\frac{\Theta_{x}}{2\sqrt{\Theta}}\,\theta^{\eta}$. Therefore
\([e_{s},e_{\eta}]=-\frac{\Theta_{x}}{2\sqrt{\Theta}}\,e_{\eta}\in\mathcal{V}.\)

$(2)$ Computation of $A=-\II_\eta$. Since
\[
\varpi_{\eta(k,\beta)}=-\frac{q_{k}\sqrt{\Theta}}{2d_{k}}\,J^{k}_{\beta\gamma}\,\theta^{k,\gamma}.
\]
Then
\[
\theta^{\eta}([e_{k,\alpha},e_{k,\beta}])
=-\frac{q_{k}\sqrt{\Theta}}{2d_{k}}J^{k}_{\beta\alpha}
+\frac{q_{k}\sqrt{\Theta}}{2d_{k}}J^{k}_{\alpha\beta}
=\frac{q_{k}\sqrt{\Theta}}{d_{k}}J^{k}_{\alpha\beta},
\]
where we used $J^{k}_{\beta\alpha}=-J^{k}_{\alpha\beta}$.

A similar computation gives $\theta^{s}([e_{k,\alpha},e_{k,\beta}])=0$. Hence
\[
A_{e_{k,\alpha}}e_{k,\beta}=\frac{1}{2}P_{\cV}[e_{k,\alpha},e_{k,\beta}]
=\frac{q_{k}\sqrt{\Theta}}{2d_{k}}J^{k}_{\alpha\beta}\,e_{\eta}.
\] 
When $k\neq j$, the $1$-forms $\varpi_{\eta(j,\beta)}$ and $\varpi_{\eta(k,\alpha)}$ are supported on disjoint coframe indices. Consequently, $\theta^{\eta}([e_{k,\alpha},e_{j,\beta}])=0$, and hence $A_{e_{k,\alpha}}e_{j,\beta}=0$.

$(3)$ Computation of $\II_s$. 
By definition of $\II$, 
\[
\II_s(e_{k,\alpha},e_{j,\beta})=\langle\nabla_{e_{k,\alpha}}e_{s},e_{j,\beta}\rangle e_s=\varpi_{(j,\beta)s}(e_{k,\alpha})e_s
=\frac{q_{k}\sqrt{\Theta}}{2d_{k}}\,\delta_{\alpha\beta}\,\delta_{kj}e_s.
\]

$(4)$ Notice that $\sum_{\alpha,\beta}(J^{k}_{\alpha\beta})^{2}=2n_{k}$. Hence
\(|A|^{2}=\sum_{k=0}^{r}\sum_{\alpha,\beta=1}^{2n_{k}}\Bigl|\frac{q_{k}\sqrt{\Theta}}{2d_{k}}J^{k}_{\alpha\beta}\Bigr|^{2}=\sum_{k=0}^{r}\frac{n_{k}q_{k}^{2}\Theta}{2d_{k}^{2}}.\)

A similar computation shows that $|\mathrm{II}_s|^{2}=\sum_{k=0}^{r}\frac{n_{k}q_{k}^{2}\Theta}{2d_{k}^{2}}$.
\end{proof}

\subsection{Geometry along the K\"ahler--Ricci flow}
Fix the initial metric
\[
\omega_0=\sum_{k=1}^r a_{k,0}\, \omega_k+\ddbar \varphi_0(\rho)\in B_0c_1(H)+\pi^*\sum_{k=1}^r a_{k,0}[\omega_k]
\]
with $B_0,a_{k,0}\in\mathbb Q$. We now discuss the geometry along the K\"ahler--Ricci flow starting from $\omega_0$.

Suppose \begin{equation}
\omega(t)=\sum_{k=1}^r a_{k}(t)\, \omega_k+\ddbar \varphi(\rho,t)\label{omega}
\end{equation}
for $t\in[0,T)$ with $\omega(0)=\omega_0$. Then
\[
\begin{aligned}
\ddt{\omega(t)}&=\sum_{k=1}^r a'_{k}(t)\, \omega_k+\ddbar \varphi_t(\rho,t),\\
\Ric(\omega(t))&=\sum_{k=1}^{r} (\lambda_k-mq_k)\omega_k-\ddbar(\log{p\Theta}-m\rho).
\end{aligned}
\]
Hence $\omega(t)$ solves the K\"ahler--Ricci flow if and only if the functions $a_k$ satisfy
\begin{equation}\frac{\dif a_k(t)}{\dif t}=-(\lambda_k-mq_k)\label{a}
\end{equation}
and $\varphi$ satisfies
\begin{equation}\begin{aligned}
\varphi_t&=\log{p\Theta}-m\rho+c(t)\\
&=(m-1)\log \varphi_\rho+\sum_{k=1}^r n_k\log (a_k(t)+q_k\varphi_\rho) +\log \varphi_{\rho\rho}-m\rho+c(t),\label{phi}
\end{aligned}\end{equation}
where $c(t)$ is a normalization constant chosen such that $\varphi_t(0,t)=0$.

Differentiating \eqref{phi} with respect to $\rho$ and changing variables from $\rho$ to $x$, we obtain
\begin{equation}\label{x}
\dd_tx=\Theta_x+\left(\frac{m-1}{x}+\sum_{k=1}^r\frac{n_kq_k}{d_k}\right)\Theta-m;
\end{equation}

Differentiating once more, we obtain
\begin{equation} \label{Theta}
\Theta_t=\Theta\Theta_{xx}+(m-\Theta_x)\Theta_x-\left(\frac{m-1}{x^2}+\sum_{k=1}^r \frac{n_kq_k^2}{d_k^2} \right)\Theta^2.
\end{equation}
Here $d_k(x,t)=a_k(t)+q_kx$.

By standard parabolic theory, equation \eqref{Theta} admits a unique solution, as does equation \eqref{phi}. The form $\omega(t)$ defined by $a_k(t)$ and $\varphi(\rho,t)$ satisfying equations \eqref{a} and \eqref{phi}, respectively, is therefore the unique solution to the K\"ahler--Ricci flow. In particular, the Calabi ansatz constructed in Section~\ref{ansatz} is preserved by the flow.

Now consider the solution $\omega(t)$ to the K\"ahler--Ricci flow defined by \eqref{omega}. The solution of ODE \eqref{a} is given by $a_k(t)=a_{k,0}-(\lambda_k-mq_k)t$. The positivity condition in Proposition~\ref{kahler} shows that
\[
a_k(t)>0,\quad b_k(t)=a_k(t)+q_kB(t)>0.
\]
On the one hand, from \eqref{class} and \eqref{chern}, we obtain
\[
[\omega(t)]=[\omega_0]-tc_1(X)=(B_0-(m+1)t)c_1(H)+\sum_{k=1}^r (a_{k,0}-(\lambda_k-mq_k)t)[\omega_k]
\].
On the other hand, \eqref{class} gives
\[
[\omega(t)]=B(t)c_1(H)+\sum_{k=1}^r a_k(t)[\omega_k]
\]
since $\omega(t)$ is of the Calabi ansatz \eqref{calabi}.

Hence $B(t)=B_0-(m+1)t$, and the positivity of $[\omega(t)]$ implies that $B(t)>0$.

Consequently, the maximal existence time $T$ is the first time at which one of $a_k$, $b_k$, or $B$ reaches $0$; in particular, $T$ is finite. More precisely, define
\[
T_B=\frac{B_0}{m+1}, T_{a,k}=\begin{cases}\frac{a_{k,0}}{\lambda_k-mq_k},& \lambda_k-mq_k>0; \\
+\infty,& \lambda_k-mq_k\leq 0,
\end{cases}, T_{b,k}=\begin{cases}\frac{a_{k,0}+q_kB_0}{\lambda_k+q_k},& \lambda_k+q_k>0; \\
+\infty,& \lambda_k+q_k\leq 0.
\end{cases}
\]
Then
\[
T=\min\{T_B, T_{a,k}, T_{b,k}: k=1,\cdots,r\}<+\infty.
\]

\subsection{Estimates of Metric Components}

We collect some basic estimates for the metric coefficient functions.

The next proposition and the following corollaries give estimates for $\Theta$. Adapting \cite[Proposition 10.2]{JST23}, we use the multiples of the function $\frac1B x(B-x)$ to trap $\Theta$, which  is the unique quadratic polynomial that has the same values and \(x\)-derivatives as \(\Theta\) at the endpoints $0$ and \(B(t)\).

\begin{prop}
There exists $C\geq 1$ such that
\[
\frac{1}{C}\frac{x(B(t)-x)}{B(t)}\leq \Theta\leq C\frac{x(B(t)-x)}{B(t)}
\]
for all $t\in(0,T)$ and $x\in (0,B(t))$.\label{G}
\end{prop}

\begin{proof}
Set $G(x,t)=\frac{B(t)\Theta(x,t)}{x(B(t)-x)}$. By the boundary conditions, we can extend it continuously by setting
\[
G(0,t)=G(B(t),t)=1.
\]
Let $H=\log G$. Then \begin{align*}
H_t&=-\frac{B_tx}{B(B-x)}+\frac{\Theta_t}{\Theta}=\frac{(m+1)x}{B(B-x)}+\frac{\Theta_t}{\Theta};\\
H_x&=\frac{\Theta_x}{\Theta}-\frac{B-2x}{x(B-x)};\\
H_{xx}&=\frac{\Theta_{xx}}{\Theta}-\left(\frac{\Theta_x}{\Theta}\right)^2+\frac{1}{x^2}+\frac{1}{(B-x)^2};
\end{align*}

In particular, if $H_x=0$, i.e. \(\frac{\Theta_x}{\Theta}=\frac{B-2x}{x(B-x)}\), then
\[
H_{xx}=\frac{\Theta_{xx}}{\Theta}+\frac{2}{x(B-x)}.
\]

Hence, combining the formulas above with \eqref{Theta}, at an interior spatial extremum of $H$ we have
\[
(\dd_t-\Theta\dd_{xx})H=\frac{1}{B}\left[\frac{m(B-x)}{x}+\frac{x}{B-x}\right](1-G)-\Theta\sum_{k=1}^r\frac{n_kq_k^2}{d_k^2(x)}.
\]

Notice that $(\dd_t-\Theta\dd_{xx})H<0$ when $G>1$. By the maximum principle (see Remark~\ref{mp}),
\[
G\leq \max\{1,\sup\{G(x,0):x\in(0,B_0)\}\}.
\]

Now denote\[
n_+=\sum_{k:\, q_k>0} n_k,\quad n_-=\sum_{k:\, q_k<0} n_k
\]

If $q_k>0$, since $d_k=a_k+q_kx\geq q_kx$, we have\[
\Theta\frac{n_kq_k^2}{d^2_k}\leq \Theta\frac{n_k}{x^2}=G\frac{n_k(B-x)}{Bx};
\]

If $q_k<0$, since $b_k=a_k+q_kB\geq 0$, we have\[
\Theta\frac{n_kq_k^2}{d^2_k}\leq \Theta\frac{n_k}{(B-x)^2}=G\frac{n_kx}{B(B-x)};
\]

Therefore, \[
(\dd_t-\Theta\dd_{xx})H\geq (m-(m+n_+)G)\frac{B-x}{Bx}+(1-(1+n_-)G)\frac{x}{B(B-x)}.
\]

Set $c_0=\left(1+\max\left\{\frac{n_+}{m},n_-\right\}\right)^{-1}>0$. Then $(\dd_t-\Theta\dd_{xx})H>0$ when $G<c_0$. By the maximum principle,
\[
G\geq \min\{c_0,\inf\{G(x,0): x\in(0,B_0)\}\}.
\]

\end{proof}

\begin{rmk}\label{mp}
Strictly speaking, we cannot use the maximum principle when the spatial interval $x\in[0,B(t)]$ moves with time. However, in our case, introducing the normalized coordinate
\[
\widetilde{x}=\frac{x}{B(t)}\in [0,1]
\]
can transform the domain into the fixed parabolic cylinder $[0,1]\times[0,T)$. For any $C^{2}$ function $H(x,t)$, define $\widetilde H(\widetilde{x},t)=H(B(t)\widetilde{x},t)$. The chain rule gives
\begin{equation}\label{eq:chain-rule}
\partial_{t}\widetilde H\Big|_{\widetilde{x}}
=\partial_{t}H\Big|_{x}+\frac{B'(t)\widetilde{x}}{B(t)}\,\widetilde H_{\widetilde{x}},
\quad
\widetilde H_{\widetilde{x}\widetilde{x}}=B(t)^{2}H_{xx}.
\end{equation}
At a spatial extremum point of $\widetilde H$, we have $\widetilde H_{\widetilde{x}}=0$. Thus
\[
\partial_{t}\widetilde H\Big|_{\widetilde{x}}=\partial_{t}H\Big|_{x},
\quad
\operatorname{sgn}(\widetilde H_{\widetilde{x}\widetilde{x}})=\operatorname{sgn}(H_{xx}).
\]
Consequently, at such a point the sign of $(\partial_{t}-\Theta\partial_{xx})H$ computed at fixed $x$ coincides with the sign of $\left(\partial_{t}-\frac{\Theta}{B(t)^{2}}\partial_{\widetilde{x}\widetilde{x}}\right)\widetilde H$ computed at fixed $\widetilde{x}$. Hence the preceding maximum-principle argument in the $x$-variable is justified by this change of variables.

In what follows, all maximum-principle arguments will be carried out at fixed $x$ without further mention of $\widetilde{x}$.
\end{rmk}

\begin{coro}\label{Theta0}
There exists $C>0$ such that\[
\Theta\leq Cx,\quad \left|\frac{q_k\Theta}{d_k}\right|\leq C
\] for all $t\in[0,T)$ and $x\in [0,B(t)]$.\label{C0Theta}
\end{coro}

\begin{proof}

Proposition~\ref{G} gives
\[
\frac{x(B-x)}{CB}\leq \Theta \leq \frac{Cx(B-x)}{B}\leq Cx.
\]
If $q_k>0$, since $d_k=a_k+q_kx\geq q_kx$, we have\[
0<\frac{q_k\Theta}{d_k}\leq \frac{\Theta}{x}\leq \frac{C(B-x)}{B}\leq C.
\]If $q_k<0$, since $b_k=a_k+q_kB\geq 0$, we have\[
0<-\frac{q_k\Theta}{d_k}\leq \frac{\Theta}{B-x}\leq \frac{Cx}{B}\leq C.
\]

\end{proof}

\begin{coro}\label{Theta1}
There exists $C>0$ such that
\[
|\Theta_{x}|\leq C
\]
for all $t\in[0,T)$ and $x\in [0,B(t)]$.
\end{coro}

\begin{proof}
Set $V=\Theta_x$. Differentiating \eqref{Theta} with respect to $x$, we obtain
\begin{equation}\label{V}
V_t=\Theta V_{xx}+(m-V)V_x+2\Theta\left[\frac{m-1}{x^{2}}\left(\frac{\Theta}{x}-V\right)+\sum_{k=1}^r \frac{n_{k} q_{k}^{2}}{d_{k}^{2}}\left(\frac{q_{k} \Theta}{d_{k}}-V\right)\right].
\end{equation}
Notice that
\[
V(0,t)=1, V(B(t),t)=-1.
\]

By Proposition~\ref{G} and Corollary~\ref{C0Theta}, there exists $C_0>0$ such that
\[
\left|\frac{q_k\Theta}{d_k}\right|\leq C_0,\quad 0<\frac{\Theta}{x}\leq C_0
\]
for all $t\in[0,T)$ and $x\in(0,B(t))$. Choose $C_1>\max\{C_0,1,\sup\{|V(x,0)|: {x\in(0,B_0)}\}\}$.

Assume that $V$ reaches the value $C_1$ for the first time at $(x_0,t_0)\in (0,B(t_0))\times(0,T)$. Since $C_1>1=V(0,t_0)>V(B(t_0),t_0)$, the point $x_0$ lies in the interior. At $(x_0,t_0)$ we then have
\[
V_x=0, V_{xx}\leq 0, V_t\geq 0.
\]

However, formula \eqref{V} gives \(V_t<0\), a contradiction. Thus $V<C_1$ for all $t\in[0,T)$ and $x\in[0,B(t)]$. A similar argument gives $V>-C_1$ throughout the same domain. Hence $|V|=|\Theta_x|\leq C_1$.

\end{proof}

The following simple but important proposition gives a uniform linear lower bound for $d_k$.

\begin{prop}\label{dk}
There exists $C>0$ such that\[
d_k(x,t)\geq C(T-t)
\]
for all $k\in\{1,\cdots,r\}$, $t\in[0,T)$, and $x\in[0,B(t))$.
\end{prop}

\begin{proof}
Denote $y=\frac{x}{T-t}$ and $D_k(y,t)=\frac{d_k((T-t)y,t)}{T-t}$.

Notice that \[
D_k(y,t)=\frac{a_k(t)}{T-t}+q_ky
\]
is a linear function of $y$, and $a_k(t),D_k(y,t)>0$ for $y\in\left[0,\frac{B(t)}{T-t}\right)$ by the positivity condition. We consider the two possible signs of $q_k$.

$(1)$ $q_k>0$. We need to verify that $\lim_{t\to T}D_k(0,t)>0$.

In this case, since $a_k(T)\geq 0$, we obtain \[
D_k(0,t)=\frac{a_k(T)}{T-t}+(\lambda_k-mq_k)\geq \frac{a_k(T)}{T}+(\lambda_k-mq_k)=\frac{a_{k,0}}{T}>0,
\]
Thus $\liminf_{t\to T}D_k(0,t)\geq\frac{a_{k,0}}{T}>0$.

$(2)$ $q_k<0$. A similar argument shows that \(\liminf_{t\to T}D_k\left(\frac{B(t)}{T-t},t\right)>0.\)

This proves the proposition.
\end{proof}

\section{Proof of Type~I Singularity}\label{proof}

In this section, we prove Theorem~\ref{thm} by treating separately the two possible cases: the fiber-collapse case $B(T)=0$ and the contraction case $B(T)>0$.

\subsection{Basic Setup}

Before the proof, we introduce some notational conventions.

We define the index sets
\[\begin{aligned}
S_a=\{k: a_k(T)=0\},&\quad S_b=\{k: b_k(T)=0\},\\
S=S_a\cup S_b,&\quad S^C=\{1,\cdots,r\}\setminus S.
\end{aligned}\]
For an index set $I\subset \{1,\cdots,r\}$, denote $N_{I}=\prod_{k\in I} N_k$ and let $\pi_I:X\to N_{I}$ be the corresponding projection.

Since $a_k$, $b_k$, $B$, and $d_k$ are linear functions, the following identities will be used repeatedly.

$(1)$ If $k\in S_a$, i.e. $a_k(T)=a_{k,0}-(\lambda_k-mq_k)T=0$, then $a_k(0)>0$ implies that $\lambda_k-mq_k>0$ and $a_{k,0}=(\lambda_k-mq_k)T$. For the same reason, $k\in S_b$ implies that $\lambda_k+q_k>0$ and $a_{k,0}=(\lambda_k+q_k)T-q_kB_0$.

$(2)$ If $B(T)=0$, i.e. $B_0=(m+1)T$, then\[
b_k(t)=a_k(T)+(\lambda_k+q_k)(T-t),
\]
which implies that $a_k(T)=0$ if and only if $b_k(T)=0$, i.e. $S=S_a=S_b$.

For parabolic rescaling, let \(\widetilde{g}=\frac{1}{T-t}g,\, y=\frac{x}{T-t}, \) and
\[\Psi(y,t)= \frac{\Theta((T-t)y,t)}{T-t},\, D_k(y,t)=\frac{d_k((T-t)y,t)}{T-t},\, L(t)=\frac{B(t)}{T-t}.
\]
Then the rescaled metric has the form
\[
\widetilde g(t)=\frac{\dif y^2}{\Psi(y,t)}+\Psi(y,t)\eta^2+yg_{FS}+\sum_{k=1}^r D_k(y,t)g_k.
\]
with scalar curvature
\begin{equation}\label{tRR}\begin{aligned}
\widetilde\RR=(T-t)R&=\sum_{k=1}^r \frac{n_k\lambda_k}{D_k}+\frac{m(m-1)}{y}-\frac{(P\Psi)_{yy}}{P}\\
&=\sum_{k=1}^r \frac{n_k\lambda_k}{D_k}+\frac{m(m-1)}{y}-\Psi_{yy}-2\frac{P_y}{P}\Psi_y-\frac{P_{yy}}{P}\Psi,
\end{aligned}
\end{equation}
where $P(y,t)=y^{m-1}\prod_{k=1}^r D_k^{n_k}$. Moreover, Proposition~\ref{G} is equivalent to the existence of a constant $C\geq 1$ such that
\begin{equation}
\frac{1}{C}\frac{y(L(t)-y)}{L(t)}\leq \Psi(y,t) \leq C\frac{y(L(t)-y)}{L(t)}.\label{Psi}
\end{equation}
Corollaries~\ref{Theta0} and~\ref{Theta1} then show that there exists $C>0$ such that
\[
|\Psi|\leq Cy,\quad \left|\frac{q_k\Psi}{D_k}\right|\leq C,\quad |\Psi_y|\leq C.
\]
Proposition~\ref{dk} gives a uniform positive lower bound for $D_k(y,t)$; that is, there exists $C>0$ such that $D_k(y,t)\geq C$.

Notice that if $B(T)=B_0-(m+1)T=0$, then \[
L(t)\equiv\frac{(m+1)(T-t)}{T-t}=m+1<+\infty
\] 
and if $B(T)>0$, then
\[
L(T)=\lim_{t\to T} \frac{B_0-(m+1)t}{T-t}=+\infty;
\]

\subsection{Curvature Estimates near $P_0$}

After the rescaling, we have \[
\widetilde{\RR}=(T-t)\RR ,\, |\Rm_{\tilde{g}(t)}|_{\tilde{g}(t)}=(T-t)|\Rm_{g(t)}|_{g(t)}.
\] Hence $\Rm(g(t))$ satisfies a Type~I estimate if and only if $\Rm(\tilde{g}(t))$ is uniformly bounded. 

The rescaled Riemann curvature tensor takes the following form.

\begin{lemma}\label{tRm}
For all $k,j=0,\cdots,r$, with the Greek indices ranging over the dimensions corresponding to their factor indices, write the components of $\Rm(\tilde{g}(t))$ as $\widetilde{\RR}_{ABCD}$. Then
\begin{align*}
\widetilde\RR_{\eta s s\eta}&=\frac12\Psi_{yy},\\
\widetilde\RR_{(k,\alpha)s\,s(k,\beta)}&=-\widetilde{P}_k\delta_{\alpha\beta},\quad\widetilde\RR_{(k,\alpha)s\,\eta(k,\beta)}=\widetilde{P}_kJ^k_{\alpha\beta},\quad \widetilde\RR_{\eta(k,\alpha)\,\eta(k,\beta)}=\widetilde{P}_k\delta_{\alpha\beta},\\
\widetilde\RR_{(k,\alpha)(k,\beta)(k,\gamma)(k,\delta)}&=\frac{1}{D_k}\,\RR^k_{\alpha\beta\gamma\delta}-\widetilde{Q}_{kk}\Bigl(\delta_{\alpha\gamma}\delta_{\beta\delta}-\delta_{\alpha\delta}\delta_{\beta\gamma}+J^k_{\alpha\gamma}J^k_{\beta\delta}-J^k_{\alpha\delta}J^k_{\beta\gamma}+2J^k_{\alpha\beta}J^k_{\gamma\delta}\Bigr),\\
\widetilde\RR_{(k,\alpha)(j,\beta)(k,\gamma)(j,\delta)}&=-\widetilde{Q}_{kj}\Bigl(\delta_{\alpha\gamma}\delta_{\beta\delta}+J^k_{\alpha\gamma}J^j_{\beta\delta}\Bigr), k\ne j,
\end{align*}
where $\widetilde{P}_k=\frac{q_k^2\Psi}{4D_k^2}-\frac{q_k\Psi_y}{4D_k}$ and $\widetilde{Q}_{kj}=\frac{q_kq_j\Psi}{4D_kD_j}$. All other components are either zero or determined by the Bianchi symmetries
\[
\widetilde\RR_{ABCD}=-\widetilde\RR_{ABDC}=-\widetilde\RR_{BACD}=\widetilde\RR_{CDAB},
\]
\end{lemma}

\begin{lemma}\label{Rmbound}
Let \(Y\in (0,L(t)]\) be finite. Suppose that
$\widetilde\RR$ is bounded on $\{0\le y\le Y\}\times[0,T)$.
Then there exists $C_Y>0$ such that
\[
\sup_{\{0\le y\le Y\}\times[0,T)}
|\Rm(\widetilde g(t))|_{\widetilde g(t)}
\le C_Y.
\]
\end{lemma}

\begin{proof}
By Lemma~\ref{tRm}, there exists $C>0$ such that
\[\begin{aligned}
|\Rm(\widetilde{g}(t))|_{\widetilde{g}(t)}\leq C&\left[|\Psi_{yy}|+\frac{|\Psi_y-1|}{y}+\frac{|\Psi-y|}{y^2}\right.\\
&\left.+\sum_{k=1}^r\left(\frac{|\Rm^k|}{D_k}+\frac{|q_k\Psi_y|}{D_k}+\left(\frac{q_k^2}{D_k^2}+\frac{|q_k|}{yD_k}+\sum_{j=1}^r\frac{|q_kq_j|}{D_kD_j}\right)\Psi\right)\right].
\end{aligned}\]

Since $D_k$ has a uniform positive lower bound by Proposition~\ref{dk}, $|\Rm(\widetilde{g}(t))|_{\widetilde{g}(t)}$ is uniformly bounded once we prove that
\[
|\Psi_{yy}|,\, \frac{|\Psi_{y}-1|}{y},\, \frac{|\Psi-y|}{y^2}
\]
are uniformly bounded.

Let $F=\Psi-y$. Define\[\begin{aligned}
W(y,t)&=\frac{P(y,t)}{y^{m-1}}=\prod_{k=1}^r D_k(y,t)^{n_k},\\
E(y,t)&=\sum_{k=1}^r\frac{n_k\lambda_k}{D_k}-\widetilde\RR-2m\frac{W_y}{W}-y\frac{W_{yy}}{W}\\
&=F_{yy}+2\frac{P_y}{P}F_y+\frac{P_{yy}}{P}F.
\end{aligned}
\]
Then formula \eqref{tRR} is equivalent to
\begin{equation}\label{PF}
(PF)_{yy}=PE.
\end{equation}

Solving the ODE \eqref{PF} with $F(0,t)=F_y(0,t)=0$, we have\[
F(y,t)=\int_{0}^y(y-s)\frac{P(s,t)}{P(y,t)}E(s,t)\dif s,
\]\[
F_y(y,t)=\int_0^y\frac{P(s,t)}{P(y,t)}E(s,t)\dif s-\frac{P_y(y,t)}{P(y,t)}F(y,t),
\]\[
F_{yy}(y,t)=E(y,t)-2\frac{P_y}{P}F_y-\frac{P_{yy}}{P}F.
\]

Using again the uniform positive lower bound for $D_k$, a straightforward computation shows that

\begin{equation*}
\begin{aligned}
\frac{W_y}W=\sum_{k=1}^r\frac{n_kq_k}{D_k},\quad \frac{W_{yy}}W=\left(\sum_{k=1}^r\frac{n_kq_k}{D_k}\right)^2-\sum_{k=1}^r\frac{n_kq_k^2}{D_k^2},\quad \frac{P_y}{P}=\frac{m-1}{y}+\sum_{k=1}^r\frac{n_kq_k}{D_k},\\
\frac{P_{yy}}P=\frac{(m-1)(m-2)}{y^2}+\frac{2(m-1)}y\sum_{k=1}^r\frac{n_kq_k}{D_k}+\left(\sum_{k=1}^r\frac{n_kq_k}{D_k}\right)^2-\sum_{k=1}^r\frac{n_kq_k^2}{D_k^2}.
\end{aligned}
\end{equation*}
These formulas imply that $\frac{W_{yy}}{W}$ and $\frac{W_y}{W}$ are uniformly bounded and that $\frac{P_y}{P}=O\left(y^{-1}\right)$ and $\frac{P_{yy}}{P}=O(y^{-2})$. Moreover, $\frac{P(s,t)}{P(y,t)}$ is uniformly bounded for $y\in[0,Y]$ and $s\in[0,y]$. Combining these estimates with the uniform bound for $\widetilde\RR$, we see that $|E|$ is also uniformly bounded. Consequently, there exists $C_Y>0$ such that
\[
|F|=|\Psi-y|\leq C_Yy^2,\, |F_y|=|\Psi_y-1|\leq C_Yy,\, |F_{yy}|=|\Psi_{yy}|\leq C_Y.
\]

\end{proof}

\begin{rmk}\label{limitation}
The estimates in Lemma~\ref{Rmbound} do not extend directly to $Y=\infty$. First, the preceding argument does not provide the required bounds for $W$, $P$, and hence $\Psi$ on $y\in[0,+\infty)$. Moreover, the constants $C_Y$ may diverge as $Y\to\infty$, so this argument does not yield a bound independent of $Y$.

\end{rmk}

\subsection{Fiber Collapse ($B(T)=0$)}
In this section, we consider the case $B(T)=0$, i.e. $T=T_B$.

When $B(T)=0$, we have $L(t)\equiv m+1$, so the range of $y$ is exactly $[0,m+1]$ for all $t\in[0,T)$. Moreover,\[
a_{k}(t)=(\lambda_k-mq_k)(T-t),\, b_k(t)=(\lambda_k+q_k)(T-t),
\] for each $k\in S=S_a=S_b$.

The goal of this section is the following theorem.
\begin{thm}\label{thmB0}
If $B(T)=0$, then there exists $C>0$ such that\[
\sup_{X\times[0,T)}(T-t)|\Rm(g(t))|\leq C.
\]
\end{thm}

Since $L(T)=m+1$ is finite, Lemma~\ref{Rmbound} gives a global bound for $\widetilde\Rm$ once we establish a global scalar curvature bound.

\begin{lemma}\label{R_B0}
If $B(T)=0$, then there exists $C>0$ such that\[
\sup_{X\times[0,T)}|\widetilde\RR(g(t))|\leq C.
\]
\end{lemma}
\begin{proof}
For $z\in N_{S^C}$, define the fiber $F_z=\pi^{-1}_{S^C}(z)\cong \BP(\cO_{N_S}\oplus L_{S}^{\oplus m})$, where $L_S=\bigotimes_{k\in S} \pi^*_{N_k}L_k$. The induced metric is\[
\widetilde{g}(t)|_{F_z}=\frac{1}{\Psi}\dif y^2+\Psi \eta^2+yg_{FS}+\sum_{k\in S}D_k(y,t)g_k
\]

Notice that $y=0$ on $P_0$. Then
\[
\widetilde{g}(t)|_{P_0\cap F_z}=\sum_{k\in S}D_k(0,t)g_k=\sum_{k\in S}(\lambda_k-mq_k)g_k.
\]
Hence
\[\begin{aligned}
\diam(F_z,\widetilde{g}(t))
&\leq 2\int_0^{m+1}\frac{\dif y}{\sqrt{\Psi(y,t)}}+\diam (N_{S},\sum\limits_{k\in S}(\lambda_k-mq_k)g_k)\\
&\leq 2\pi\sqrt{(m+1)C_0}+\diam (N_{S},\sum\limits_{k\in S}(\lambda_k-mq_k)g_k),
\end{aligned}\]
where $C_0>0$ is the constant in formula \eqref{Psi}. Hence the diameter of $F_z$ is bounded uniformly in $z$ and $t$. Denote this uniform bound by $D$.

Now fix $t\in[0,T)$ and $\theta\in (\Phi^*)^{-1}{\zeta}$. Let $p_t$ be the Ricci vertex associated with $\theta$. Set $z_t=\pi_{S^C}(p_t)$. Then $p_t\in F_{z_t}\subset B_{\widetilde{g}(t)}(p_t,D)$, and Lemma~\ref{JST1.3} shows that
\[
|\widetilde\RR(q,t)|\leq C_1(1+D^2)
\]
on $B_{\widetilde{g}(t)}(p_t,D)$. Therefore,
\[
\sup_{F_{z_t}}|\widetilde\RR(q,t)|\leq C,
\]
where $C=C_1(1+D^2)>0$.

Since $\widetilde\RR$ depends only on $y$ and $t$ by formula \eqref{tRR}, and since $C$ is independent of $t$, the bound $|\widetilde\RR(q,t)|\leq C$ holds at every point $q\in X$ and every time $t\in[0,T)$. Indeed, the $y$-coordinate on $F_{z_t}$ ranges over the entire interval $[0,m+1]$. Thus, for every $q\in X$, we can choose $q'\in F_{z_t}$ such that $y(q)=y(q')$, and then $|\widetilde\RR(q,t)|=|\widetilde\RR(q',t)|\leq C$.
\end{proof}

\begin{proof}[Proof of Theorem~\ref{thmB0}]
Since $\sup_{X\times[0,T)}|\widetilde\RR|\le C$, Lemma~\ref{Rmbound} gives
\[
\sup_{X\times[0,T)}(T-t)|\Rm_{g(t)}|_{g(t)}=\sup_{X\times[0,T)}|\Rm_{\widetilde g(t)}|_{\widetilde g(t)}\le C.
\]

\end{proof}

\subsection{Contraction ($B(T)>0$)}

In this section, we consider the case $B(T)>0$, i.e. $T=T_{a,k}$ or $T=T_{b,k}$ for some $k$.

When $B(T)>0$, we have $L(T)=+\infty$, so the range of $y$ at time $t\in[0,T)$ is $[0,L(t))$, with $L(t)\to+\infty$ as $t\to T$. If $k\in S_a$, then \[
\lambda_k-mq_k>0,\, a_{k,0}=(\lambda_k-mq_k)T;
\] If $k\in S_b$, then \[\lambda_k+q_k>0,\,a_{k,0}=(\lambda_k+q_k)T-q_kB_0.\]

The goal of this section is the following theorem.
\begin{thm}\label{thmB}
If $B(T)>0$, then there exists $C>0$ such that\[
\sup_{X\times[0,T)}(T-t)|\Rm(g(t))|\leq C.
\]
\end{thm}

This case is more difficult than the case $B(T)=0$. As stated in Remark~\ref{limitation}, Lemma~\ref{Rmbound} applies only for finite $Y$. When $B(T)>0$, the coordinate $y$ can tend to infinity, so even a global scalar curvature bound would not yield a global bound for $\widetilde\Rm$ by that lemma; in fact, no such scalar curvature bound is available here. We therefore divide the proof of Theorem~\ref{thmB} into two parts.

\paragraph*{\bfseries Part I: Near the Zero Section}

The following proposition gives the required Type~I curvature bound on every fixed bounded interval of $y$ and will exclude blow-up sequences that remain near the zero section.

\begin{prop}\label{left}
Assume $B(T)>0$. For every finite $Y\in[0,\infty)$, there exists $C_Y>0$ independent of $t$ such that
\[
\sup_{\{0\leq y \leq Y\}\times[0,T)}(T-t)|\Rm(g(t))|\leq C_Y.
\] 
\end{prop}

To apply Lemma~\ref{JST1.3} uniformly in this bounded region, we need to find a Ricci vertex that remains over a fixed neighborhood of $\Phi(P_0)$.

\begin{lemma}\label{vertex}
Fix \(q\in\Phi(P_0)\). Then there exist $\theta_q\in (\Phi^*)^{-1}\zeta$ and open neighborhoods $U_0,U_1$ of $q$ with $U_0\Subset U_1$ such that, for every \(t<T\), there exists a Ricci vertex \(p_t\) associated with \(\theta_q\) satisfying
\[
\theta_{q}|_{U_1}=0,\quad p_t\in\Phi^{-1}(U_0).
\]

\end{lemma}

\begin{proof}
First, we construct $\theta_q$. Let \(A_T\) be the base-point-free line bundle defining
\(\Phi:X\to\widetilde X\subset\mathbb CP^M\). Since \(B(T)>0\), the restriction of \(A_T\) to each projective fiber has positive degree. The corresponding space of sections decomposes according to the homogeneous degree in the fiber coordinates. On \(P_0\), only the degree-zero coordinate block can be nonzero, whereas on \(D_\infty\), only the highest-degree coordinate block can be
nonzero.

Consequently, there exists a projective hyperplane \(H_\infty\subset\mathbb CP^M\)
such that \(\Phi(D_\infty)\subset H_\infty\) and \(q\notin H_\infty\).
After a projective change of coordinates, we may write \(H_\infty=\{Z_0=0\}\) and \(q=[1:0:\cdots:0]\).

Let \(\theta_0=\frac{1}{c_T}\omega_{\mathrm{FS}}\big|_{\tilde X}\), where $c_T$ is determined by $c_T\zeta=c_1(A_T)$. Then \([\Phi^*\theta_0]=\zeta\). On the affine chart \(\mathcal U=\mathbb CP^M\setminus H_\infty=\{Z_0\ne0\}\), define
\[
h=\frac{1}{c_T}\log\frac{\sum_{j=0}^{M}|Z_j|^2}{|Z_0|^2}.
\]
Then \(h\ge0\), \(h(q)=0\), and \(\sqrt{-1}\partial\bar\partial h=\theta_0\) on \(\widetilde X\cap\mathcal U\). Moreover, \(h\to+\infty\) as one approaches \(H_\infty\).

Let \(R>0\) be a constant to be determined below. Choose a smooth nondecreasing function \(\chi_R:[0,\infty)\longrightarrow\mathbb R\) such that
\[
\chi_R(s)=s\quad\text{for }s\le2R,
\]
\[
\chi_R(s)=3R
\quad\text{for }s\ge4R,
\]

Define \(\eta_q=\chi_R(h)\) on \(\widetilde X\cap\mathcal U\). Since \(h\to+\infty\) near \(H_\infty\) and \(\chi_R\) is constant on \([4R,\infty)\), the function \(\eta_q\) extends smoothly
across \(\widetilde X\cap H_\infty\) by setting it equal to \(3R\) there. Thus \(\eta_q\) is a globally defined smooth function on \(\widetilde X\).

Set
\[
U_0=\{h<R\}\cap\widetilde X,\quad U_1=\{h<2R\}\cap\widetilde X.
\]
The properness of \(h\) toward \(H_\infty\) gives
\[
q\in U_0\Subset U_1\subset\widetilde X\setminus\Phi(D_\infty).
\]
Furthermore,
\[
\eta_q(q)=0,\quad\eta_q\ge R
\quad\text{on }\widetilde X\setminus U_0.
\]

Define
\[
\theta_q=\theta_0-\sqrt{-1}\partial\bar\partial\eta_q.
\]
Since \(\eta_q\) is globally defined,
\[
[\Phi^*\theta_q]=[\Phi^*\theta_0]=\zeta.
\]
On \(U_1\), one has \(h<2R\), and hence
\(\eta_q=\chi_R(h)=h\). Thus
\[
\theta_q|_{U_1}=\theta_0-\sqrt{-1}\partial\bar\partial h=0.
\]

It remains to locate the corresponding Ricci vertex. Let
\(u_0(\cdot,t)\) be a weighted Ricci potential associated with
\(\theta_0\), with its time-dependent additive constant chosen
so that
\[
|u_0(\cdot,t)|
\le
\frac{C_0}{T-t}
\]
for some \(C_0<\infty\) independent of \(t\). This normalization is available from the construction of the basic weighted Ricci potential in \cite[Section 4.2]{JST23}. 

Up to a function of time \(t\), the weighted Ricci potential associated with \(\theta_q\) is
\[
u_q=u_0+\frac{\Phi^*\eta_q}{T-t}.
\]
Indeed, we can verify directly that
\[
\Ric(\omega(t))-\frac{1}{T-t}\omega(t)=-\frac{1}{T-t}\Phi^*\theta_q-\sqrt{-1}\partial\bar\partial u_q.
\]

Notice that \(\eta_q(q)=0\). Therefore
\[
\inf_{\Phi^{-1}(q)}u_q(\cdot,t)
\le
\frac{C_0}{T-t}.
\]
On the other hand, on \(X\setminus\Phi^{-1}(U_0)\) one has
\[
u_q\ge-\frac{C_0}{T-t}+\frac{R}{T-t}=\frac{R-C_0}{T-t}.
\]

We now choose \(R>2C_0\). Then no global minimum of \(u_q(\cdot,t)\) can lie in \(X\setminus\Phi^{-1}(U_0)\). Hence any minimum point \(p_t\) of \(u_q(\cdot,t)\) satisfies
\[
p_t\in\Phi^{-1}(U_0).
\]
By definition, \(p_t\) is a Ricci vertex associated with \(\theta_q\).

\end{proof}

\begin{lemma}\label{pt}
There exists $C>0$ such that, for every $t\in[0,T)$,
\[
y(p_t,t)\leq C,
\]
where $p_t$ is the Ricci vertex defined in Lemma~\ref{vertex}.
\end{lemma}

\begin{proof}
We prove that $y(p_t,t)\leq C$ by contradiction. Assume that there exists a sequence $\{t_i\}_{i=1}^\infty$ such that $t_i\nearrow T$ but $y(p_{t_i},t_i)\to\infty$ as $i\to\infty$. Without loss of generality, assume that $p_{t_i}\notin P_0$, i.e. $y(p_{t_i},t_i)>0$, for every $i$.

Notice that $\theta_q$ defined in the proof of Lemma~\ref{vertex} is $U(m)$-invariant, and so is the weighted Ricci potential $u$. Indeed, $u$ is defined by
\[
-\ddbar u=\Ric (\widetilde\omega(t))-\widetilde\omega(t)+\frac{1}{T-t}\Phi^*\theta_q.
\]
The right-hand side is $U(m)$-invariant. Hence, for every $p\in X\setminus P_0$, there exists a smooth curve $\gamma_p$ such that $\gamma_p'(0)\in \cH_0|_{p}$ and $u$ is constant along $\gamma_p$. At the minimum point $p_{t_i}$, let $V_i=\gamma_{p_{t_i}}'(0)\in\cH_0|_{p_{t_i}}$ and, without loss of generality, normalize $V_i$ to have unit length with respect to $\tilde g(t_i)$. Then
\[
0=\frac{\dif^2}{\dif s^2}(u\circ\gamma_{p_{t_i}})(0)=\Hess u(V_i,V_i)+\dif u(\nabla_{V_i}V_i)=\Hess u(V_i,V_i),
\]
where the last equality holds since $\dif u|_{p_{t_i}}=0$. The same argument also applies to $JV_i$. Thus $\ddbar u(V_i,JV_i)=0$, i.e.
\[\begin{aligned}
&\Ric_{\widetilde{g}(t_i)}(V_i,V_i)=\Ric (\widetilde\omega(t_i))(V_i,JV_i)\\
=&\,\widetilde\omega(t_i)(V_i,JV_i)-\frac{1}{T-t_i}\Phi^*\theta_q(V_i,JV_i)-\ddbar u(V_i,JV_i)=1,
\end{aligned}
\]
where the last equality holds because $\Phi^*\theta_q$ vanishes near $p_{t_i}$ and $\widetilde\omega(t_i)(V_i,JV_i)=|V_i|^2_{\tilde{g}(t_i)}=1$.

However, a direct computation gives
\[
\Ric_{\tilde{g}(t_i)}(V_i,V_i)=\left.\frac{1}{y}\left(m-{\Psi_y}-\left(\frac{m-1}{y}+\sum_{k=1}^r \frac{n_kq_k}{D_k(y,t)}\right){\Psi}\right)\right|_{(p_{t_i},t_i)}.
\]
Together with Proposition~\ref{G} and Corollaries~\ref{Theta0}--\ref{Theta1}, this gives
\[
\Ric_{\tilde{g}(t_i)}(V_i,V_i)=O(y(p_{t_i},t_i)^{-1}),
\]
a contradiction. Consequently, there exists $C>0$ such that $y(p_t,t)\leq C$.

\end{proof}

\begin{proof}[Proof of Proposition~\ref{left}]

Fix $t\in [0,T)$ and let $p_t$ be the Ricci vertex defined in Lemma~\ref{vertex}. For any \(z\in[0,Y]\) in the range of $y(\cdot,t)$, choose a point $q_t(z)$ on the same radial ray as $p_t$ such that \(y(q_t(z),t)=z\). Then, by Lemma~\ref{pt} and Proposition~\ref{G}, there exist $C_0,C_{Y,0}>0$ such that
\[
d_{\widetilde g(t)}(p_t,q_t(z))\leq\int_0^{y(p_t,t)}\frac{\dif \sigma}{\sqrt{\Psi(\sigma,t)}}+\int_0^z\frac{\dif \sigma}{\sqrt{\Psi(\sigma,t)}}\leq2\sqrt{2C_{Y,0}}+2\sqrt{2C_0\,Y}.
\]
Lemma~\ref{JST1.3} then shows that there exists $C_Y>0$ such that
\[
\left|\widetilde \RR(q_t(z),t)\right|\leq C\left(1+d_{\widetilde g(t)}^2(q_t(z),p_t)\right)\leq C_Y.
\]

Since $\widetilde\RR$ depends only on $y$ and $t$, every point $q\in X$ satisfying $y(q,t)=z$ has
\(\widetilde\RR(q,t)=\widetilde\RR(q_t(z),t)\).
It follows that
\[
\sup_{\{0\leq y\leq Y\}\times[0,T)}
|\widetilde\RR|
\leq C_Y.
\]

By Lemma~\ref{Rmbound}, the Riemann curvature bound holds on $\{0\leq y \leq Y\}\times[0,T)$.

\end{proof}

\paragraph*{\bfseries Part II: Far away from the Zero Section}
In the region where $y$ tends to infinity, we rule out a Type~II singularity by identifying the geometry of the singularity limit model.

The following proposition excludes Type~II blow-up along every sequence escaping to $y=+\infty$ and thus provides the estimate complementary to Proposition~\ref{left}.

\begin{prop}\label{right}
Assume $B(T)>0$. Let $\{(q_i,t_i)\}_{i=1}^\infty$ be a sequence in $X\times[0,T)$ with $t_i\nearrow T$ and $y(q_i,t_i)\to+\infty$ as $i\to\infty$. Then
\[
\limsup_{i\to\infty} \,(T-t_i)\,|\mathrm{Rm}(g(t_i))|_{g(t_i)}(q_i)<+\infty.
\]
\end{prop}

We organize the proof by contradiction. Suppose that the conclusion fails. By Hamilton's Type~II point-picking argument \cite{Ham95b}, after replacing the original sequence, we may choose $(q_i,t_i)$ and
$K_i=|\operatorname{Rm}(g(t_i))|_{g(t_i)}(q_i)$
such that
\[t_i\nearrow T,\quad K_i(T-t_i)\to+\infty,\quad y(q_i,t_i)\to+\infty.\]

Define
\[g_i(\tau)=K_i g\left(t_i+K_i^{-1}\tau\right),\quad\tau\in[-K_it_i,K_i(T-t_i)).\]
The points may be chosen so that $|\operatorname{Rm}(g_i(0))|_{g_i(0)}(q_i)=1$
and, for every fixed $A>0$,
\[\limsup_{i\to\infty}\sup_{X\times[-A,A]}|\operatorname{Rm}(g_i(\tau))|_{g_i(\tau)}\leq1.\]

\begin{rmk}\label{gi}
Here $g_i$ denotes the rescaled metric and should not be confused with the fixed metrics $g_k$ on the base factors.

\end{rmk}

For $R,T'>0$ and $q\in X$, define the parabolic ball by
\[
PB_{g}(q;R,T')=\bigcup_{\tau\in[-T',T']}B_{g(\tau)}(q,R)\times\{\tau\},
\]
where $B_{g}(U,R)$ denotes the tubular neighborhood $\bigcup_{p\in U} B_{g}(p,R)$ of $U\subset X$ with respect to the metric $g$.

The following two lemmas prepare for the construction of parallel flat horizontal distributions on the limit space.

\begin{lemma}\label{inf}
For any $R,T'>0$,
\[
\inf_{PB_{g_i}(q_i;R,T')} K_ix(q,t_i+K_i^{-1}\tau)\to\infty,
\]\[
\inf_{PB_{g_i}(q_i;R,T')} K_id_k(x(q,t_i+K_i^{-1}\tau),t_i+K_i^{-1}\tau)\to\infty
\]
as $i\to\infty$.

In particular, for all sufficiently large $i$, $PB_{g_i}(q_i;R,T')\subset (X\setminus P_0)\times[-K_it_i,K_i(T-t_i))$.
\end{lemma}

\begin{proof}

For convenience, set
\[
x_i(q,\tau)=x(q,t_i+K_i^{-1}\tau),\, y_i(q,\tau)=y(q,t_i+K_i^{-1}\tau),
\]\[
d_{k,i}(q,\tau)=d_k(x_i(q,\tau),t_i+K_i^{-1}\tau).
\]

The Type~II condition implies that $K_i(T-t_i)\to+\infty$ and $y_i(q_i,0)\to\infty$ as $i\to\infty$. Hence
\[
K_ix_i(q_i,0)=K_i(T-t_i)y_i(q_i,0),\quad K_id_{k,i}(q_i,0)=K_i(T-t_i)D_k(y_i(q_i,0),t_i),
\]
and both quantities diverge to $+\infty$ as $i\to\infty$.

It remains to prove uniform divergence on the whole parabolic ball $PB_{g_i}(q_i;R,T')$.
First, consider the center $q_i$. Combining formula \eqref{x} with Corollaries~\ref{Theta0} and~\ref{Theta1}, we find a constant $C_0>0$ such that
\[
|\dd_{\tau} x_i|=|K_i^{-1} \dd_t x|\leq \frac{C_0}{K_i},\quad |\dd_{\tau} d_{k,i}|=|K_i^{-1} \dd_t d_k|\leq \frac{C_0}{K_i}.
\]
Hence, uniformly for $\tau\in[-T',T']$,
\[\begin{aligned}
K_ix_i(q_i,\tau)&=K_ix_i(q_i,0)+\int_0^\tau K_i\dd_{\sigma} x_i(q_i,\sigma)\dif \sigma \\
&\geq K_ix_i(q_i,0)-C_0T'\to \infty,\\
K_id_{k,i}(q_i,\tau)&=K_id_{k,i}(q_i,0)+\int_0^\tau K_i\dd_{\sigma} d_{k,i}(q_i,\sigma)\dif \sigma\\
& \geq K_id_{k,i}(q_i,0)-C_0T'\to \infty.
\end{aligned}\]

Next, fix $\tau\in[-T',T']$. By Corollary~\ref{Theta0}, there exists $C_1>0$ independent of $\tau$ such that
\[\begin{aligned}
|\nabla^{g_i(\tau)}\log x_i|^2&=K_i^{-1}|\nabla^{g(t_i+K_i^{-1}\tau)}\log x_i|^2\leq \frac{C_1}{K_ix_i},\\
|\nabla^{g_i(\tau)}\log d_{k,i}|^2&=K_i^{-1}|\nabla^{g(t_i+K_i^{-1}\tau)}\log d_{k,i}|^2\leq \frac{C_1}{K_id_{k,i}}.
\end{aligned}\]
Consider a smooth curve $\gamma$ starting from $q_i$, parametrized by $g_i(\tau)$-arc length, whose length is less than $R$. Assume that $\gamma(s_0)$ is the first point at which the $x_i$-coordinate reaches $\frac12x_i(q_i,\tau)$ or $2x_i(q_i,\tau)$. Thus $x_i(\gamma(s_0),\tau)=\frac12x_i(q_i,\tau)$ or $2x_i(q_i,\tau)$, while $x_i(\gamma(s),\tau)\in (\frac12x_i(q_i,\tau),2x_i(q_i,\tau))$ for all $s<s_0$. Then
\[
\left|\log \frac{x_i(\gamma(s_0),\tau)}{x_i(q_i,\tau)}\right|\leq\int_{0}^{s_0}|\nabla^{g_i(\tau)}\log x_i|\dif s\leq R\sqrt{\frac{2C_1}{K_ix_i(q_i,\tau)}}\to 0
\]
as $i\to\infty$. This contradicts the definition of $s_0$ for all sufficiently large $i$. Hence $x_i\in (\frac12x_i(q_i,\tau),2x_i(q_i,\tau))$ on $B_{g_i(\tau)}(q_i,R)$, and therefore
\[
\inf_{B_{g_i(\tau)}(q_i,R)}K_ix\geq \frac{K_ix_i(q_i,\tau)}{2}\to+\infty\quad\text{as }i\to\infty.
\]
The same argument applies to $\inf_{B_{g_i(\tau)}(q_i,R)}K_id_{k,i}$.

Finally, since the preceding estimates are uniform for $\tau\in[-T',T']$, we obtain
\[
\inf_{PB_{g_i}(q_i;R,T')} K_ix_i(q,\tau)\to\infty,\quad \inf_{PB_{g_i}(q_i;R,T')} K_id_{k,i}(q,\tau)\to\infty.
\]

\end{proof}

\begin{lemma}\label{P}
The horizontal projection $P_{\cH}$ on $TX_0$ extends smoothly to $T(X\setminus P_0)$. Moreover, for every $K>0$ and $t<T$, there exists $C>0$, independent of $K$ and $t$, such that
\[\begin{aligned}
|\nabla^{Kg(t)}P_{\cH}|^2_{Kg(t)}& \leq C\left(\frac{1}{Kx}+\sum_{k=1}^r\frac{1}{Kd_k}\right),\\
|\Rm_{Kg(t)}|_{\cH}& \leq C\left(\frac{1}{Kx}+\sum_{k=1}^r\frac{1}{Kd_k}\right),
\end{aligned}\]
on $X\setminus P_0$.

\end{lemma}

\begin{proof}

First, we work on $TX_0$. By definition,
\[
\nabla_X P_{\cH}(Y)=\nabla_X(P_\cH Y)-P_\cH(\nabla_X Y).
\]

If $Y\in\cH$, then $\nabla_X P_{\cH}(Y)=(\mathrm{Id}-P_\cH)(\nabla_X Y)=P_\cV(\nabla_X Y)$, while if $Y\in\cV$, then $\nabla_X P_{\cH}(Y)=-P_\cH(\nabla_X Y)$. Thus, by Corollary~\ref{oneill},
\[\begin{aligned}
|\nabla^{Kg(t)} P_\cH|^2_{Kg(t)} & =\frac{2}{K}\sum_{k=0}^r\sum_{\alpha,\beta=1}^{2n_k}\left(\langle\nabla_{e_{k,\alpha}}e_s,e_{k,\beta}\rangle^2+\langle\nabla_{e_{k,\alpha}}e_\eta,e_{k,\beta}\rangle^2\right)\\
& =\frac{2}{K}\left(\frac{(m-1)\Theta}{x^2}+\sum_{k=1}^{r}\frac{n_kq_k^2\Theta}{d^2_k}\right)\leq C\left(\frac{1}{Kx}+\sum_{k=1}^r\frac{1}{Kd_k}\right).
\end{aligned}\]

The formulas for the Riemann curvature tensor on $\cH$ in Corollary~\ref{Rm} give
\[\begin{aligned}
|\Rm_{Kg(t)}|_{\cH}& \leq \frac{1}{K}|\Rm_{g(t)}|_{\cH} \leq \frac{C_0}{K}\left(\sum_{k=1}^{r}\frac{|\Rm^k|}{d_k}+\frac{1}{x}+\frac{\Theta}{x^2}+\sum_{k=1}^r\frac{|q_k|\Theta}{xd_k}+\sum_{k,j=1}^r \frac{|q_kq_j|\Theta}{d_kd_j} \right)\\
&\leq C\left(\frac{1}{Kx}+\sum_{k=1}^r\frac{1}{Kd_k}\right).
\end{aligned},
\]
for some constants $C_0,C>0$, where the last inequality follows from Proposition~\ref{G} and Corollary~\ref{Theta0}.

It remains to check that $P_\cH$ extends smoothly across $D_\infty$. Once it extends smoothly, the estimates above remain valid near $D_\infty$.

Consider the normal bundle $E=N_{D_\infty/X}\cong \pi^*L^{-1}\otimes\pi_0^*\cO_{\BC P^{m-1}}(1)$, which has complex rank $1$. Choose local holomorphic coordinates $(z,w)$, where $z=(z^1,\cdots,z^{n+m-1})$ are coordinates on $D_\infty$ and $w\in \BC$ is the fiber coordinate, so that $D_\infty=\{w=0\}$. Denote by \(\Gamma=\Gamma_\mu(z)\dif z^{\mu}\) the Chern connection $1$-form of the induced Hermitian metric on $E$, where $\Gamma_\mu\in C^{\infty}(D_\infty)$. The horizontal lift of $\dd_{z^\mu}$ is given by $X^{\cH}_\mu=\dd_{z^\mu}-\Gamma_\mu(z)w\dd_w$. In particular, $X^\cH_\mu|_{w=0}=\dd_{z^\mu}$.

Now $P_\cH$ extends smoothly to $T(X\setminus P_0)$ by $P_\cH(X^\cH_\mu)=X^\cH_\mu$ and $P_\cH(\dd_w)=P_\cH(\bar\dd_{w})=0$. In particular, $P_\cH|_{D_\infty}$ is exactly the projection onto $TD_\infty$.
\end{proof}

By Perelman's local noncollapsing theorem \cite{Per02} and Hamilton's compactness theorem \cite{Ham95a}, after passing to a subsequence, the pointed Ricci flows \((X,g_i(\tau),q_i)\) converge smoothly in the pointed Cheeger--Gromov sense to an eternal Ricci flow
\[(X_\infty,g_\infty(\tau),q_\infty), \qquad -\infty<\tau<\infty.\]
The limit is complete, nonflat, and has bounded curvature. It satisfies
$$
|\operatorname{Rm}(g_\infty(\tau))|_{g_\infty(\tau)}\leq1,\quad|\operatorname{Rm}(g_\infty(0))|_{g_\infty(0)}(q_\infty)=1.
$$

Now we can prove that the universal cover of $X_\infty$ splits as a product of a two-dimensional nonflat factor and a flat factor.

\begin{lemma}
The universal cover of \(X_\infty\) splits isometrically as
\[
(\widetilde X_\infty,\widetilde g_\infty(\tau))
\cong
(\Sigma^2,h(\tau))
\times
(\mathbb R^{2m+2n-2},g_{\mathrm{Euc}}),
\]
where \((\Sigma^2,h(\tau))\) is Hamilton's cigar soliton up to scaling
and pullback by diffeomorphisms.

\end{lemma}

\begin{proof}
Hamilton's compactness theorem \cite{Ham95a} gives a sequence of embeddings $F_i:U_i\to X$, with $U_i\nearrow X_\infty$ and $q_i=F_i(q_\infty)$, such that $F_i^*g_i(\tau)\to g_\infty(\tau)$ smoothly on every compact subset of $X_\infty\times\BR$. In particular, $F_i^*g_i(0)$ is equivalent to $g_\infty(0)$; that is, there exists $C_i>0$ such that $C_i^{-1}g_\infty(0)\leq F_i^*g_i(0)\leq C_ig_\infty(0)$.

Fix $l\in \BN$ and set $B_l=\overline{B_{g_\infty(0)}(q_\infty;l)}$. Then $F_i(B_l)\subset X\setminus P_0$ for all sufficiently large $i$. Combining Lemmas~\ref{inf} and~\ref{P}, we obtain
\[
\sup_{F_i(B_l)\times[-l,l]}|\nabla^{g_i(\tau)}P_\cH|\to 0,\quad \sup_{F_i(B_l)\times[-l,l]}|\Rm_{g_i(\tau)}|_\cH\to 0
\]
as $i\to \infty$.

Define $\hat{P}_i=F_i^*P_\cH$. Then
\[
\sup_{B_l\times[-l,l]} |\nabla^{F_i^*g_i(\tau)}\hat{P}_i|_{F_i^*g_i(\tau)}\to 0\quad\text{as }i\to\infty,\qquad \dd_\tau \hat{P}_i\equiv 0.
\]

Next, we construct the orthogonal projection ${P}_\infty\in\operatorname{End}(TX_\infty)$. The sequence $\hat{P}_i(q_\infty)$ is bounded with respect to $g_\infty(0)$. Hence, after passing to a subsequence, $\hat{P}_{i}(q_\infty)$ converges to some $P_\infty(q_\infty)\in \operatorname{End}(T_{q_\infty}X_\infty)$.

For any $p\in B_l$, let $\gamma:[0,1]\to B_l$ be a smooth curve from $q_\infty$ to $p$, and let $\PT_{i,\gamma}$ be parallel transport along $\gamma$ with respect to $F_i^*g_i(0)$. Then
\[\begin{aligned}
||\hat{P}_i(p)-\PT_{i,\gamma}\hat{P}_i(q_\infty)\PT_{i,\gamma}^{-1}||&\leq \int_0^1 ||(\nabla^{F_i^*g_i(0)}_{\gamma'(s)} \hat{P}_i)(\gamma(s))||\cdot|\gamma'(s)|\dif s\\
&\leq L_{i,\gamma}|\nabla^{F_i^*g_i(0)}\hat{P}_i|_{F_i^*g_i(0)},
\end{aligned}\]
where $L_{i,\gamma}$ is the length of $\gamma$ with respect to $F_i^*g_i(0)$. Since $F_i^*g_i(0)\to g_\infty(0)$ on $B_l$, $\PT_{i,\gamma}$ converges to parallel transport $\PT_\gamma$ with respect to $g_\infty(0)$. Hence, for each $p\in B_l$, define $P_\infty(p)=\PT_\gamma P_\infty(q_\infty)\PT_{\gamma}^{-1}$. Then $\hat{P}_i(p)\to P_\infty(p)$. To check that $P_\infty(p)$ is well-defined, suppose that $\gamma_1$ and $\gamma_2$ are two curves in $B_l$ connecting $q_\infty$ to $p$. Then
\[\begin{aligned}
&||\PT_{i,\gamma_1}\hat{P}_i(q_\infty)\PT_{i,\gamma_1}^{-1}-\PT_{i,\gamma_2}\hat{P}_i(q_\infty)\PT_{i,\gamma_2}^{-1}||\\
\leq& ||\PT_{i,\gamma_1}\hat{P}_i(q_\infty)\PT_{i,\gamma_1}^{-1}-\hat{P}_i(p)||+||\hat{P}_i(p)-\PT_{i,\gamma_2}\hat{P}_i(q_\infty)\PT_{i,\gamma_2}^{-1}||
\end{aligned}\]
Letting $i\to\infty$, we obtain $\PT_{\gamma_1}P_\infty(q_\infty)\PT_{\gamma_1}^{-1}=\PT_{\gamma_2}P_\infty(q_\infty)\PT_{\gamma_2}^{-1}$. Since $l\in \BN$ is arbitrary, $P_\infty$ can be constructed on all of $TX_\infty$.

Notice that \[
P_\infty^2=P_\infty,\, \nabla^{g_\infty(\tau)}P_\infty=0,\, \rank P_\infty(p)=\dim \Im P_\infty(p)=2m+2n-2.
\]
Set $\cH_\infty=\Im P_\infty$ and $\cV_\infty=\Ker P_\infty$. Then we obtain the orthogonal decomposition $TX_\infty=\cH_\infty\oplus\cV_\infty$ and $|\Rm_{g_\infty(\tau)}|_{\cH_\infty}=0$. By the de Rham splitting theorem, the universal cover is $\tilde{X}_\infty\cong\Sigma^2\times\BR^{2(m+n-1)}$ with $\tilde{g}_\infty(\tau)=h(\tau)\oplus g_{\operatorname{Euc}}$.

Since $|\mathrm{Rm}(g_\infty(0))|(q_\infty)=1$ and the Euclidean factor is flat,
$\Sigma^2$ is nonflat. By B.-L.~Chen's theorem \cite{Chen09},
any complete ancient solution to the Ricci flow with bounded curvature has
nonnegative scalar curvature; the two-dimensional strong maximum principle
then yields $\RR_\Sigma>0$ on $\Sigma^2\times\mathbb R$. Because
$|\mathrm{Rm}(g_\infty)|\le 1$ and equality holds at an interior space-time
point, $\RR_\Sigma$ attains a global positive maximum there. Hamilton's
differential Harnack rigidity \cite{Ham93} implies that $(\Sigma^2,h(\tau))$ is a simply connected steady
gradient Ricci soliton. By Hamilton's classification in dimension two \cite{Ham88}, $\Sigma^2$ is Hamilton's cigar soliton.

\end{proof}

\begin{proof}[Proof of Proposition~\ref{right}]
Perelman's local noncollapsing theorem \cite{Per02} implies that the blow-up limit is $\kappa$-noncollapsed at all scales. Therefore, the product of the cigar soliton and a flat space cannot be a singularity model. This contradiction proves Proposition~\ref{right}.
\end{proof}

\begin{proof}[Proof of Theorem~\ref{thmB}]
Suppose, towards a contradiction, that the conclusion fails. Then there
exists a sequence $(q_i,t_i)\in X\times[0,T)$ with $t_i\nearrow T$ such that
\begin{equation*}
(T-t_i)\,|\mathrm{Rm}(g(t_i))|_{g(t_i)}(q_i)\to+\infty.
\end{equation*}

If $y(q_i,t_i)$ were bounded, then passing to a subsequence we would have
$y(q_i,t_i)\le Y$ for some fixed $Y<\infty$ and all $i$. By
Proposition~\ref{left},
\[
\sup_{\{0\le y\le Y\}\times[0,T)}(T-t)\,|\mathrm{Rm}(g(t))|_{g(t)}\le C_Y<\infty,
\]
which is a contradiction. Hence, after passing to a further
subsequence, \(y(q_i,t_i)\to+\infty\).

However, Proposition~\ref{right} gives
\[
\limsup_{i\to\infty}(T-t_i)\,|\mathrm{Rm}(g(t_i))|_{g(t_i)}(q_i)<+\infty,
\]
This contradiction shows that no such sequence exists. Hence the global Type~I bound holds when $B(T)>0$.
\end{proof}

\begin{proof}[Proof of Theorem~\ref{thm}]
Combining Propositions~\ref{thmB0} and~\ref{thmB} completes the proof.
\end{proof}

\end{document}